\pdfoutput=1 

\documentclass[11pt]{amsart}

\usepackage{amssymb,amsthm,enumitem,colonequals,tikz-cd,microtype}
\usepackage[normalem]{ulem} 

\usepackage[osf]{Baskervaldx}
\usepackage[baskervaldx]{newtxmath}
\usepackage[cal=boondoxo]{mathalfa} 

\usepackage[top=3.75cm, bottom=3cm, left=3.5cm, right=3.5cm]{geometry}

\usepackage{xcolor}
\colorlet{darkblue}{blue!55!black}
\colorlet{darkcyan}{cyan!50!black}
\colorlet{darkgreen}{green!60!black}

\PassOptionsToPackage{hyphens}{url}
\usepackage{hyperref}
\hypersetup{
  colorlinks=true,
  linkcolor=darkblue,
  urlcolor=darkcyan,
  citecolor=darkgreen,
}

\def\eqref#1{\textcolor{darkblue}{(\ref{#1})}}

\usepackage[nameinlink]{cleveref} 
\Crefformat{section}{#2\S#1#3}
\Crefmultiformat{section}{#2\S\S#1#3}{ and~#2#1#3}{, #2#1#3}{, and~#2#1#3}

\usepackage[pagewise]{lineno}
\let\oldequation\equation
\let\oldendequation\endequation
\renewenvironment{equation}{\linenomathNonumbers\oldequation}{\oldendequation\endlinenomath}
\expandafter\let\expandafter\oldequationstar\csname equation*\endcsname
\expandafter\let\expandafter\oldendequationstar\csname endequation*\endcsname
\renewenvironment{equation*}{\linenomathNonumbers\oldequationstar}{\oldendequationstar\endlinenomath}
\let\oldalign\align
\let\oldendalign\endalign
\renewenvironment{align}{\linenomathNonumbers\oldalign}{\oldendalign\endlinenomath}
\expandafter\let\expandafter\oldalignstar\csname align*\endcsname
\expandafter\let\expandafter\oldendalignstar\csname endalign*\endcsname
\renewenvironment{align*}{\linenomathNonumbers\oldalignstar}{\oldendalignstar\endlinenomath}

\theoremstyle{plain}
\newtheorem{theorem}{Theorem}[section]
\newtheorem{lemma}[theorem]{Lemma}
\newtheorem{corollary}[theorem]{Corollary}
\newtheorem{proposition}[theorem]{Proposition}

\theoremstyle{definition}
\newtheorem{definition}[theorem]{Definition}
\newtheorem{example}[theorem]{Example}
\newtheorem{remark}[theorem]{Remark}
\newtheorem{setup}[theorem]{Setup}

\newtheorem{convention}[theorem]{Convention}
\newtheorem{reminder}[theorem]{Reminder}

\newtheorem*{ack}{Acknowledgments} 

\AddToHook{env/conjecture/begin}{\crefalias{theorem}{conjecture}}
\AddToHook{env/lemma/begin}{\crefalias{theorem}{lemma}}
\AddToHook{env/corollary/begin}{\crefalias{theorem}{corollary}}
\AddToHook{env/proposition/begin}{\crefalias{theorem}{proposition}}
\AddToHook{env/definition/begin}{\crefalias{theorem}{definition}}
\AddToHook{env/remark/begin}{\crefalias{theorem}{remark}}
\AddToHook{env/setup/begin}{\crefalias{theorem}{setup}}
\AddToHook{env/example/begin}{\crefalias{theorem}{example}}
\AddToHook{env/terminology/begin}{\crefalias{theorem}{terminology}}
\AddToHook{env/convention/begin}{\crefalias{theorem}{convention}}
\AddToHook{env/reminder/begin}{\crefalias{theorem}{reminder}}

\numberwithin{equation}{section}
\numberwithin{theorem}{section}

\title[Proxy smallness meets $t$-structures]{Proxy smallness meets $t$-structures}

\author[M.~Hrbek]{Michal Hrbek}
\address{M.~Hrbek,
Institute of Mathematics of the Czech Academy of Sciences,
Žitná 25, 115 67 Prague, Czechia}
\email{hrbek@math.cas.cz}

\author[P.~Lank]{Pat Lank}
\address{P.~Lank,
Dipartimento di Matematica “F. Enriques”, Universit\`{a} degli Studi di Milano, Via Cesare
Saldini 50, 20133 Milano, Italy}
\email{plankmathematics@gmail.com}

\author[G.~Le Gros]{Giovanna Le Gros}
\address{G.~Le Gros,
Institute of Mathematics of the Czech Academy of Sciences,
Žitná 25, 115 67 Prague, Czechia}
\email{legros@math.cas.cz}

\author[S.~Pavon]{Sergio Pavon}
\address{S.~Pavon,
Dipartimento di Informatica, University of Verona, Strada le Grazie 15
37134 Verona, Italy}
\email{sergio.pavon@univr.it}

\date{\today}

\keywords{Local complete intersection singularities, proxy smallness, $t$-structures, tensor actions, triangulated categories}

\subjclass[2020]{14A30 (primary), 14F08, 13D09, 18G80, 14B05} 

\begin{document}
    
\begin{abstract}
    We introduce a notion of proxy smallness for $t$-structures on triangulated categories associated to a Noetherian scheme. Specifically, the theory is developed in the presence of tensor actions. Consequently, our results yield a new characterization of schemes that are locally complete intersections in terms of $t$-structures, as well as a topological classification of preaisles on the bounded derived category of coherent sheaves.
\end{abstract}

\maketitle

\tableofcontents

\section{Introduction}
\label{sec:intro}

\subsection{What is known}
\label{sec:intro_what_is_known}

Proxy smallness, introduced in \cite{Dwyer/Greenlees/Iyengar:2006a,Dwyer/Greenlees/Iyengar:2006b}, provides a flexible weakening of compactness in triangulated categories. It has found applications in commutative algebra and homotopy theory. More recently, it has been related to structural properties of derived categories; see e.g.\ \cite{Briggs/Iyengar/Stevenson:2025}.

A key application for this notion comes from \cite[Theorem 5.2]{Pollitz:2019}. Loc.\ cit.\ showed that a Noetherian local ring $R$ is a complete intersection if, and only if, every object of $D^b_{\operatorname{coh}}(R)$ is proxy small. This was extended in \cite[Corollary 5.7]{Letz:2021} to Noetherian rings which are locally complete intersections. Thus, proxy smallness detects singularities.

However, proxy smallness does not behave well under globalization: naive global analogues fail to capture the correct geometry when passing from local to global settings. To address this, \cite{Briggs/Iyengar/Letz/Pollitz:2022} introduced a tensor refinement. Roughly speaking, over a Noetherian scheme $X$, an object $E \in D^b_{\operatorname{coh}}(X)$ is said to be $\otimes$-proxy small if it generates, via tensor operations, a perfect complex with the same support. They show that $X$ is locally a complete intersection if, and only if, every object of $D^b_{\operatorname{coh}}(X)$ is $\otimes$-proxy small.

\subsection{What we do}
\label{sec:intro_what_we_do}
 
The goal of this paper is to show that proxy smallness admits a natural refinement when viewed through the lens of $t$-structures and tensor actions.

\subsubsection{Overview}
\label{sec:intro_what_we_do_overview}

Tensor actions between triangulated categories associated to a scheme arise naturally. For instance, in Krause's recollement relating $D_{\operatorname{qc}}(X)$ and the homotopy category of injectives $K(\operatorname{Inj}(X))$ \cite{Krause:2005}. In this setting, we introduce a compatibility notion between tensor structures and $t$-structures. See \Cref{sec:prelim_recollement} and \Cref{sec:prelim_tensor}. Our approach proceeds in three steps.

First, we study $\otimes$-aisles in $D_{\operatorname{qc}}(X)$ generated by complexes from $D^b_{\operatorname{coh}}(X)$, showing that they are compactly generated (see \Cref{thm:ALS_global_full_classification}). This extends \cite[Theorem 3.10]{AlonsoTarrio/JeremiasLopez/Saorin:2010} from the affine case to arbitrary Noetherian schemes. The result is of independent interest and provides a key technical input.

Second, we introduce a notion of proxy smallness adapted to $t$-structures. See \Cref{sec:proxy_smallness}. After developing its basic properties and comparing it with the classical notion, we characterize $t$-proxy smallness in terms of recollements \Cref{thm:characterize_left_proxy_small}, yielding a $t$-structural analogue of \cite[Theorem 4.3]{Briggs/Iyengar/Stevenson:2025}.  Continuing the comparison between these two notions, we then discuss some examples and show that the collection of $t$-proxy small objects is strictly smaller than the collection of proxy small objects \Cref{ex:proxy-small-not-t-proxy-small}.

Third, we incorporate tensor actions into this framework. We introduce a notion of $t$-$\otimes$-proxy small objects (see \Cref{def:t_proxy_tensor_small}) and study subcategories that are stable under both the tensor action and the operations defining preaisles. 

\subsubsection{$t$-$\otimes$-proxy small objects}
\label{sec:intro_what_we_do_t_proxy_small}

We now introduce the central notion of this paper. Let $\mathcal{T}$ be a rigidly compactly generated tensor triangulated category. Fix a preaisle $\mathcal{P}^{\leq 0}$ in $\mathcal{T}^c$. For $P \in \mathcal{T}$, we write $\langle P \rangle^{(-\infty,0]}_\otimes$ for the smallest subcategory of $\mathcal{T}$ containing $P$ and closed under nonnegative shifts, extensions, direct summands, and the $\otimes$-action of $\mathcal{P}^{\leq 0}$. We also write $\overline{\langle P \rangle}^{(-\infty,0]}_\otimes$ for the subcategory which is, in addition, closed under all small coproducts in $\mathcal{T}$. We say $P \in \mathcal{T}$ is \textit{$t$-$\otimes$-proxy small} if $\overline{\langle P \rangle}^{(-\infty,0]}_{\otimes}$ is compactly generated in $\mathcal{T}$ and $\overline{\langle P \rangle}^{(-\infty,0]}_{\otimes} \cap \mathcal{T}^c \subseteq \langle P \rangle^{(-\infty,0]}_{\otimes}$. Moreover, in any triangulated category $\mathcal{K}$, we set $\langle \mathcal{S} \rangle^{(-\infty,0]}$ to be the smallest preaisle generated by a subcategory $\mathcal{S}\subseteq \mathcal{K}$.

\subsubsection{Structural result}
\label{sec:intro_what_we_do_structural}

We discuss our main structural result. We study preaisles $\mathcal{A} \subseteq D^b_{\operatorname{coh}}(X)$ satisfying the following conditions:
\begin{itemize}
    \item $\mathcal{A}$ is closed under direct summands
    \item $\mathcal{A}$ is closed under tensoring with objects from $\operatorname{Perf}^{\leq 0}(X) = \operatorname{Perf}(X) \cap D^{\leq 0}_{\operatorname{qc}}(X)$
    \item $\mathcal{A}$ is generated by complexes from $D^b_{\operatorname{coh}}(X)$ that are $t$-$\otimes$-proxy small objects of $D_{\operatorname{qc}}(X)$ (with respect to the preaisle $\mathcal{P}^{\leq 0} = \operatorname{Perf}^{\leq 0}(X)$).
\end{itemize}
The following identifies the data governing such subcategories.

\begin{theorem}
    \label{introthm:injective_mapping_for_t_proxy_tensor_aisles_scheme}
    There is an injective function $\Theta$ mapping the collection of preaisles $\mathcal{A}$ above to the Cartesian product of 
    \begin{itemize}
        \item preaisles of $D_{\operatorname{sg}}(X)$ closed under direct summands and the action of $\operatorname{Perf}^{\leq 0}(X)$
        \item Thomason filtrations on $X$ (i.e.\ certain functions from $\mathbb{Z}$ to specialization closed subsets of $X$; see \Cref{sec:prelim_Thomason}).
    \end{itemize}
    It is given by the assignment
    \begin{displaymath}
        \mathcal{A}\mapsto \bigg( \langle \pi (\mathcal{A}) \rangle^{(-\infty,0]} , \phi_{\mathcal{A}} \bigg)
    \end{displaymath}
    where $\phi_{\mathcal{A}}$ is the associated Thomason filtration of the compactly generated $\otimes$-aisle $\overline{\langle \mathcal{A} \rangle}^{(-\infty,0]}_{\otimes}$ in $D_{\operatorname{qc}}(X)$ obtained by \Cref{thm:ALS_global_full_classification}  and $\pi\colon D^b_{\operatorname{coh}}(X) \to D_{\operatorname{sg}}(X)$ the Verdier localization functor.
\end{theorem}

This result is proved in \Cref{thm:injective_mapping_for_t_proxy_tensor_aisles_scheme}. The result relates $t$-$\otimes$-proxy small subcategories of $D^b_{\operatorname{coh}}(X)$ to data arising from the singularity category together with the topology of $X$. Although $\Theta$ is not a bijection in general, we give in \Cref{prop:image_assignment_tstr_general} a criterion describing its image.

\subsubsection{Characterization \& classification}
\label{sec:intro_what_we_do_characterization_classification}

We next extract geometric consequences of this framework. To start:

\begin{proposition}
    \label{prop:LCI_characterizations}
    Let $X$ be a Noetherian scheme. Then the following are equivalent:
    \begin{enumerate}
        \item \label{prop:LCI_characterizations1} $X$ is locally a complete intersection
        \item \label{prop:LCI_characterizations3} Every object of $D^b_{\operatorname{coh}}(X)$ is $\otimes$-proxy small
        \item \label{prop:LCI_characterizations2} Every object of $D^b_{\operatorname{coh}}(X)$ is $t$-$\otimes$-proxy small.
    \end{enumerate}
\end{proposition}

The key novelty lies in condition \eqref{prop:LCI_characterizations2}, which provides a $t$-structural refinement of existing characterizations. In fact, the result upgrades \cite[Theorem 6.4]{Briggs/Iyengar/Letz/Pollitz:2022} away from the separated case. The proof is partly inspired by \cite[Theorem 5.2]{Pollitz:2019}; see \Cref{prop:local-CI-tproxy}. However, the crucial step is establishing a local-to-global property, see \Cref{lem:stalk-t-exact}. Specifically, membership in a compactly generated aisle on $K(\operatorname{Inj}(X))$ is a stalk local condition.

Producing examples of proxy small objects of $D^b_{\operatorname{coh}}(X)$ which are not $t$-proxy small does not seem so straightforward, even in the affine local case. We do so in our \Cref{ex:proxy-small-not-t-proxy-small}, which uses two technical building blocks. The first is inspired by work of \cite{Dwyer/Greenlees/Iyengar:2006a}, adjusted to our $t$-structure refinement of proxy smallness, and reduces the problem to a module theoretic one. For the module theoretic problem, we utilize work of \cite{Jorgensen/Sega:2004} which in turn is influenced by \cite{Gasharov/Peeva:1990}. Notably, each of \cite{Jorgensen/Sega:2004,Gasharov/Peeva:1990} provides a counterexample to two unrelated well-known conjectures, and the common theme with our work is the exploitation of bad or unexpected behavior of modules over rings which are nice enough, yet not local complete intersection.  

Finally, we obtain classification results in situations where \Cref{thm:injective_mapping_for_t_proxy_tensor_aisles_scheme} becomes a bijection. First, in the case of hypersurface singularities.

\begin{theorem}
    \label{thm:image_assignment_tstr_lci_hyp}
    Let $X$ be a separated Noetherian scheme with only hypersurface singularities. Then $\Theta$ can be interpreted as a bijection between
    \begin{enumerate}
        \item $\otimes$-suspended subcategories of $D_{\operatorname{coh}}^b(X)$, and
        \item pairs $(W,\phi)$ of a specialization closed subset $W$ of $\operatorname{sing}(X)$ and a Thomason filtration $\phi$ on $X$ such that $W \subseteq \bigcup_{n \in \mathbb{Z}}\phi(n)$.
    \end{enumerate}
\end{theorem}

Here, $\operatorname{sing}$ is the singular locus of the scheme. As a reminder, $\otimes$-suspended subcategories of $D_{\operatorname{coh}}^b(X)$ are those closed under suspension, direct summands, and the $\operatorname{Perf}^{\leq 0}(X)$-action on $D^b_{\operatorname{coh}}(X)$. 

For the second result, we consider schemes obtained as certain zero loci:

\begin{example}
    \label{ex:intro_lci}
    Consider a separated regular Noetherian scheme $T$ of finite Krull dimension, a vector bundle $\mathcal{E}$ on $T$ of finite rank, and a global section $t \in H^0(T,\mathcal{E})$. Denote by $X$ the zero subscheme of $t$ in $T$. Assume that $X$ has the resolution property and that $\mathcal{O}_\mathcal{E}(1)$ is ample. Following \cite{Orlov:2006}, see also the generalization of \cite[Appendix A]{Burke/Walker:2015} and discussion in \cite[\S 2]{Stevenson:2014}, $X$ satisfying the conditions above admits scheme morphisms $X \xleftarrow{p} Z \xrightarrow{i} Y$ such that $Z$ is the projective bundle of the normal bundle ${\mathcal{N}}_{X/T}$ in $T$ and $p$ is the canonical projection, $Y$ is a hypersurface scheme and $i$ is a closed immersion.
\end{example}

The next result is the following:

\begin{theorem}
    \label{thm:image_assignment_tstr_ci_intro}
    Let $X$ be as in \Cref{ex:intro_lci}. Then $\Theta$ can be interpreted as a bijection between 
    \begin{enumerate}
        \item $\otimes$-suspended subcategories of $D_{\operatorname{coh}}^b(X)$, and
        \item pairs $(W,\phi)$ of a specialization closed subset $W$ of $\operatorname{sing}(Y)$ and a Thomason filtration $\phi$ on $X$ such that $p i^{-1}(W) \subseteq \bigcup_{n \in \mathbb{Z}}\phi(n)$. 
    \end{enumerate}
\end{theorem}

This appears later as \Cref{thm:image_assignment_tstr_ci}. In fact, our results provide a nonaffine generalization of Takahashi's classification \cite{Takahashi:2025}. In particular, we give a proof independent of the methods in loc.\ cit. See \Cref{rmk:generalize_takahashi}. Notably, the result is sharp in the sense of explicitness. Indeed, the construction of \Cref{ex:intro_lci} allows for one to efficiently detail the topological considerations. Moreover, \Cref{thm:image_assignment_tstr_ci_intro} should be compared to a different classification for $\otimes$-aisles on $D^b_{\operatorname{coh}}$ which appears in \cite{Clark/Lank/ManaliRahul/Parker:2024}, whereas our work proves a classification for $\otimes$-\textit{preaisles}. On a related note, classifications for notions of proxy small thick subcategories in $D^b_{\operatorname{coh}}(R)$ of a Noetherian ring have recently appeared \cite{Takahashi:2025a,Kato/Takahashi:2026}.

\begin{ack}
    Hrbek and Le Gros were supported by the project LQ100192601 Lumina quaeruntur, funded by the Czech Academy of Sciences (RVO 67985840). Lank was supported by the ERC Advanced Grant 101095900-TriCatApp. Pavon was supported by the Project 2022S97PMY \textit{Structures for Quivers, Algebras and Representations (SQUARE)} funded by NextGenerationEU under NRRP, Call PRIN 2022 No.\ 104 of February 2, 2022 of Italian Ministry of University and Research. Both Lank and Pavon would like to thank the Institute of Mathematics of the Czech Academy of Sciences for its hospitality while conducting parts of this research. Additionally, the authors thank Ryo Takahashi for discussions. 
\end{ack}

\subsection*{Notation}
\label{sec:intro_notation}

Let $X$ be a Noetherian scheme. Let $D(X)\colonequals D(\operatorname{Mod}(X))$ be the derived category of $\mathcal{O}_X$-modules. Denote by $D_{\operatorname{qc}}(X)$ the (strictly full) subcategory of $D(X)$ consisting of complexes with quasi-coherent cohomology. Set $D_{\operatorname{coh}}^b(X)$ (resp.\ $D_{\operatorname{qc}}^b(X)$) to be the (strictly full) subcategory of $D(X)$ consisting of complexes having bounded and coherent (resp.\ quasi-coherent) cohomology. Moreover, $\operatorname{Perf}(X)$ is defined as the (strictly full) subcategory of $D_{\operatorname{qc}}(X)$ consisting of the perfect complexes on $X$. At times, if $X$ is affine, we abuse notation and write $D_{\operatorname{qc}}(R)\colonequals D_{\operatorname{qc}}(X)$ where $R\colonequals H^0(X,\mathcal{O}_X)$ are the global sections; similar conventions will occur for the other categories mentioned here and later.

\section{Preliminaries}
\label{sec:prelim}

Let $\mathcal{T}$ be a triangulated category. Here, `strictly full' means a full subcategory closed under isomorphisms.

\subsection{Generation}
\label{sec:prelim_generation}

We discuss generation and dimension for triangulated categories. See \cite{Bondal/VandenBergh:2003} for details. Let $\mathcal{T}$ be a triangulated category with shift functor $[1]\colon \mathcal{T} \to \mathcal{T}$. Consider a subcategory $\mathcal{S} \subseteq \mathcal{T}$. A triangulated subcategory of $\mathcal{T}$ is called \textbf{thick} if it is closed under direct summands. Denote by $\langle \mathcal{S} \rangle$ the smallest thick subcategory of $\mathcal{T}$ containing $\mathcal{S}$; if $\mathcal{S}$ consists of a single object $G$, we write $\langle G \rangle \colonequals  \langle \mathcal{S} \rangle$. Set $\operatorname{add}(\mathcal{S})$ to be the smallest strictly full subcategory of $\mathcal{T}$ containing $\mathcal{S}$ that is closed under shifts, finite coproducts, and direct summands. Inductively, let $\langle \mathcal{S} \rangle_0$ consist of all objects in $\mathcal{T}$ isomorphic to the zero object, $\langle \mathcal{S} \rangle_1 \colonequals  \operatorname{add}(\mathcal{S})$, and 
\begin{displaymath}
    \langle \mathcal{S} \rangle_n \colonequals  \operatorname{add} \{ \operatorname{cone}(\phi) \mid \phi \in \operatorname{Hom}_{\mathcal{T}} (\langle \mathcal{S} \rangle_{n-1}, \langle \mathcal{S} \rangle_1) \}.
\end{displaymath}
It can be checked that $\langle \mathcal{S} \rangle = \cup_{n=0}^\infty \langle \mathcal{S} \rangle_n$. We say $E$ is \textbf{finitely built by $\mathcal{S}$} if $E\in \langle \mathcal{S}\rangle$.

If $\mathcal{T}$ admits small coproducts, then the collection of compact objects in $\mathcal{T}$ will be denoted by $\mathcal{T}^c$. These form a triangulated subcategory of $\mathcal{T}$. Assume that $\mathcal{T}^c$ is essentially small, then we say that $\mathcal{T}$ is \textbf{compactly generated} if it coincides with the smallest triangulated subcategory of $\mathcal{T}$ containing $\mathcal{T}^c$ and closed under small coproducts. Equivalently, $\mathcal{T}$ is compactly generated if, for any $E \in \mathcal{T}$ satisfying $\operatorname{Hom}(P, E) = 0$ for all $P \in \mathcal{T}^c$, one has $E \cong 0$. Note that classical generators for $\mathcal{T}^c$ coincide with compact generators for $\mathcal{T}$ (see e.g.\ \cite[\href{https://stacks.math.columbia.edu/tag/09SR}{Tag 09SR}]{StacksProject}). The \textbf{localizing subcategory} generated by a collection $\mathcal{S}\subseteq \mathcal{T}$, denoted by $\overline{\langle \mathcal{S} \rangle}$, is the smallest triangulated subcategory containing $\mathcal{S}$ and closed under small coproducts. In particular, $\mathcal{T} = \overline{\langle \mathcal{T}^c \rangle}$ whenever $\mathcal{T}$ is compactly generated.

\subsection{\texorpdfstring{$t$}{t}-structures}
\label{sec:prelim_$t$-structures}

We discuss $t$-structures $\mathcal{T}$ and recall material from \cite{Keller/Vossieck:1988,Beilinson/Berstein/Deligne/Gabber:2018}. A pair of strictly full subcategories $\tau = (\mathcal{T}^{\leq 0}, \mathcal{T}^{\geq 0})$ of $\mathcal{T}$ is a \textbf{$t$-structure} if:
\begin{itemize}
    \item $\operatorname{Hom}(A,B) = 0$ for all $A \in \mathcal{T}^{\leq 0}$ and $B \in \mathcal{T}^{\geq 0}[-1]$,
    \item $\mathcal{T}^{\leq 0}[1] \subseteq \mathcal{T}^{\leq 0}$ and $\mathcal{T}^{\geq 0}[-1] \subseteq \mathcal{T}^{\geq 0}$,
    \item for every $E \in \mathcal{T}$, there is a distinguished triangle
    \begin{displaymath}
        \tau^{\leq 0} E \to E \to \tau^{\geq 1} E \to (\tau^{\leq 0} E)[1]
    \end{displaymath}
    with $\tau^{\leq 0} E \in \mathcal{T}^{\leq 0}$ and $\tau^{\geq 1} E \in \mathcal{T}^{\geq 0}[-1]$.
\end{itemize}
The above triangle is unique up to unique isomorphism, and it is called the \textbf{truncation triangle} of $E$ with respect to $\tau$. Given $n \in \mathbb{Z}$, the pair $(\mathcal{T}^{\leq n}, \mathcal{T}^{\geq n})$ is also a $t$-structure on $\mathcal{T}$ where $\mathcal{T}^{\leq n} \colonequals  \mathcal{T}^{\leq 0}[-n]$ and $\mathcal{T}^{\geq n} \colonequals  \mathcal{T}^{\geq 0}[-n]$. 
Let $F \colon \mathcal{T}_1 \to \mathcal{T}_2$ be an exact functor between triangulated categories equipped with $t$-structures $(\mathcal{T}_1^{\leq 0}, \mathcal{T}_1^{\geq 0})$ and $(\mathcal{T}_2^{\leq 0}, \mathcal{T}_2^{\geq 0})$. We say that $F$ is \textbf{right $t$-exact} if $
F(\mathcal{T}_1^{\leq 0}) \subseteq \mathcal{T}_2^{\leq 0}$,
and \textbf{left $t$-exact} if 
$F(\mathcal{T}_1^{\geq 0}) \subseteq \mathcal{T}_2^{\geq 0}$. If both conditions hold, then $F$ is \textbf{$t$-exact}. 

\subsubsection{(Pre)aisles}
\label{sec:prelim_$t$-structures_pre_aisle}

A strictly full subcategory $\mathcal{A} \subseteq \mathcal{T}$ is a \textbf{preaisle} if $\mathcal{A}$ is closed under positive shifts and extensions and it is an \textbf{aisle} if, in addition, the inclusion $\mathcal{A} \to \mathcal{T}$ admits a right adjoint.
In fact, a subcategory $\mathcal{A}\subseteq\mathcal{T}$ is an aisle if and only if the pair $(\mathcal{A},\mathcal{A}^\bot[1])$ is a $t$-structure, where
\begin{displaymath}
    \mathcal{A}^\perp \colonequals  \{ T \in \mathcal{T} \mid \forall A \in \mathcal{A}, \operatorname{Hom}(A,T) = 0  \}.
\end{displaymath}
Respectively, we call $\mathcal{T}^{\leq 0}$ and $\mathcal{T}^{\geq 0}$ the \textbf{aisle} and \textbf{coaisle} of the $t$-structure. A useful fact, used freely in our work, is that aisles are closed under direct summands of coproducts when they exist in the ambient triangulated category (see \cite[Lemma 1.4]{AlonsoTarrio/Lopez/Salorio:2003}). A subcategory of $\mathcal{T}$ is called \textbf{suspended} if it is a preaisle closed under direct summands, equivalently, it is a strictly full subcategory of $\mathcal{T}$ closed under positive shifts, extensions, and direct summands. The smallest suspended category containing a subcategory $\mathcal{S} \subseteq \mathcal{T}$ is denoted $\langle \mathcal{S} \rangle^{(-\infty,0]}$; we write $\langle E \rangle^{(-\infty,0]}$ when $\mathcal{S}$ consists of a single object $E$.
Equivalently, it is the smallest preaisle containing $\mathcal{S}$ and closed under direct summands. 

\subsubsection{`Big' (pre)aisles}
\label{sec:prelim_$t$-structures_big}

Assume that $\mathcal{T}$ admits small coproducts. A preaisle $\mathcal{A} \subseteq \mathcal{T}$ is called \textbf{cocomplete} if it is closed under all coproducts in $\mathcal{T}$; note that this renders $\mathcal{A}$ a suspended subcategory by the usual trick with infinite coproducts. Given $\mathcal{S}\subseteq \mathcal{T}$, $\overline{\langle \mathcal{S} \rangle}^{(-\infty, 0]}$ is defined to be the smallest cocomplete preaisle containing $\mathcal{S}$. If $\mathcal{T}$ is well generated (of which compactly generated is a particular case), $\overline{\langle \mathcal{S} \rangle}^{(-\infty, 0]}$ is an aisle whenever $\mathcal{S}$ is essentially small, see \cite[Theorem 2.3]{Neeman:2021}. An aisle $\mathcal{U}$ on $\mathcal{T}$ is \textbf{compactly generated} when there exists a collection of compact objects $\mathcal{P} \subseteq \mathcal{T}^c$ satisfying $\overline{\langle \mathcal{P} \rangle}^{(-\infty, 0]} = \mathcal{U}$. Hence, we say a $t$-structure is \textbf{compactly generated} if its aisle is as such.  

\begin{example}
    \label{ex:standard_aisle_cpt_gen_for_schemes}
    Let $X$ be a Noetherian scheme. Then the standard aisle $D^{\leq 0}_{\operatorname{qc}}(X)$ is compactly generated by $\operatorname{Perf}(X) \cap D^{\leq 0}_{\operatorname{qc}}(X)$. See \cite{Herbera/Hrbek/LeGros:2025} for details; we briefly sketch the argument. First, one uses the equivalence $D(\operatorname{Qcoh}(X)) \cong D_{\operatorname{qc}}(X)$ 
    \cite[\href{https://stacks.math.columbia.edu/tag/09T4}{Tag 09T4}]{StacksProject}. Second, the claim may be proved by contradiction. 
    Suppose there exists a nonzero object in $D^{\geq 0}_{\operatorname{qc}}(X)$ which does not belong to the aisle generated by $\operatorname{Perf}(X) \cap D^{\leq 0}_{\operatorname{qc}}(X)$. Choose a degree in which this object has nonzero cohomology. Then the corresponding cycle complex has a nonzero component in that degree. Since quasi-coherent sheaves on a Noetherian scheme are filtered colimits of their coherent subsheaves \cite[\href{https://stacks.math.columbia.edu/tag/01PG}{Tag 01PG}]{StacksProject}, one can approximate by perfect complexes to obtain a contradiction.
\end{example}

\subsubsection{Tensor variants}
\label{sec:prelim_$t$-structures_tensor}

Assume that $\mathcal{T}$ is a tensor triangulated category with tensor $\otimes$ and unit $1$. Choose a preaisle $\mathcal{P}^{\leq 0} \subseteq \mathcal{T}$ satisfying $\mathcal{P}^{\leq 0} \otimes \mathcal{P}^{\leq 0} \subseteq \mathcal{P}^{\leq 0}$ and $1 \in \mathcal{P}^{\leq 0}$. A \textbf{tensor (pre)aisle} (or \textbf{$\otimes$-(pre)aisle}), with respect to $\mathcal{P}^{\leq 0}$, is an (pre)aisle $\mathcal{U} \subseteq \mathcal{T}$ which is closed under tensoring by $\mathcal{P}^{\leq 0}$, i.e.\ $\mathcal{P}^{\leq 0} \otimes \mathcal{U} \subseteq \mathcal{U}$.  If clear from context, we omit `with respect to $\mathcal{P}^{\leq 0}$'. We say a $t$-structure on $\mathcal{T}$ is \textbf{tensor} if its aisle is such. Similarly, a \textbf{$\otimes$-suspended} subcategory is a tensor preaisle which is closed under direct summands. Given $\mathcal{C} \subseteq \mathcal{T}$, $\langle \mathcal{C} \rangle^{(-\infty,0]}_{\otimes}$ is defined as the smallest $\otimes$-suspended subcategory containing $\mathcal{C}$ and $\overline{ \langle \mathcal{C} \rangle}^{(-\infty,0]}_{\otimes}$ is the smallest cocomplete $\otimes$-preaisle containing $\mathcal{C}$.

\subsection{Recollements}
\label{sec:prelim_recollement}

We briefly recall the notion of recollements. See \cite[\S 1.4]{Beilinson/Berstein/Deligne/Gabber:2018} for details. Recall that a \textbf{recollement} is a commutative diagram of triangulated categories and exact functors of the form 
\begin{displaymath}
    \begin{tikzcd}
        {\mathcal{T}} && {\mathcal{K}} && {\mathcal{D}}
        \arrow["I"{description}, from=1-1, to=1-3]
        \arrow["{I_\lambda}"', bend right =30pt, from=1-3, to=1-1]
        \arrow["{I_\rho}", bend right =-30pt, from=1-3, to=1-1]
        \arrow["Q"{description}, from=1-3, to=1-5]
        \arrow["{Q_\lambda}"', bend right =30pt, from=1-5, to=1-3]
        \arrow["{Q_\rho}", bend right =-30pt, from=1-5, to=1-3]
    \end{tikzcd}
\end{displaymath}
which satisfy the following properties:
\begin{itemize}
    \item $I_\lambda \dashv I \dashv I_\rho$ and $Q_\lambda \dashv Q \dashv Q_\rho$
    \item $I, Q_\lambda, Q_\rho$ are fully faithful
    \item $\ker (Q)$ coincides with the strictly full subcategory on objects of the form $I(T)$ where $T\in \mathcal{T}$.
\end{itemize}
In such a case, there are distinguished triangles
\begin{displaymath}
    \begin{aligned}
        (Q_\lambda \circ Q )(E) & \to E \to (I \circ I_\lambda )(E) \to (Q_\lambda \circ Q) (E)[1],
        \\& (I \circ I_\rho) (E) \to E \to (Q_\rho \circ Q) (E) \to (I \circ I_\rho )(E)[1]
    \end{aligned}
\end{displaymath}
which are functorial in $\mathcal{K}$. Particularly, the natural transformations between these functors are given by the (co)units of the relevant adjoint pairs. As $Q_\lambda, Q, I, I_\lambda$ are left adjoints, they preserve coproducts. Additionally, as $I$ and $Q$ admit right adjoints, $I_\lambda$ and $Q_\lambda$ preserve compact objects (see \cite[Theorem 5.1]{Neeman:1996}). 

\begin{example}
    Let $X$ be a Noetherian scheme. Then we have a recollement:
    \begin{equation}
        \label{eq:krause_recollement_schemes}
        \begin{tikzcd}
            {S_{\operatorname{qc}}(X)} && {K(\operatorname{Inj}(X))} && {D_{\operatorname{qc}}(X).}
            \arrow["I", from=1-1, to=1-3]
            \arrow["{I_\lambda}"', bend right =24pt, from=1-3, to=1-1]
            \arrow["{I_\rho}", bend right =-24pt, from=1-3, to=1-1]
            \arrow["Q", from=1-3, to=1-5]
            \arrow["{Q_\lambda}"', bend right =24pt, from=1-5, to=1-3]
            \arrow["{Q_\rho}", bend right =-24pt, from=1-5, to=1-3]
        \end{tikzcd}
    \end{equation}
    where $K(\operatorname{Inj}(X))$ is the homotopy category of injective $\mathcal{O}_X$-modules (i.e.\ of $\operatorname{Inj}(X)\subseteq \operatorname{Qcoh}(X)$) and $S_{\operatorname{qc}}(X)$ is the strictly full subcategory of $K(\operatorname{Inj}(X))$ consisting of acyclic complexes. Here, $I$ is the inclusion and $Q$ is the composite
    \begin{displaymath}
        K(\operatorname{Inj}(X)) \xrightarrow{incl.} K(\operatorname{Qcoh}(X)) \xrightarrow{can.} D_{\operatorname{qc}}(X).
    \end{displaymath}
    While we use $S_{\operatorname{qc}}(X)$, it is denoted by $S(\operatorname{Qcoh}(X))$ in \cite{Krause:2005}. Additionally, we utilize the triangulated equivalence of $D_{\operatorname{qc}}(X)$ with $D(\operatorname{Qcoh}(X))$ (see e.g.\ \cite[\href{https://stacks.math.columbia.edu/tag/09T4}{Tag 09T4}]{StacksProject}). This recollement was initially proven for \textit{separated} Noetherian schemes in \cite[Theorem 1.1]{Krause:2005}. Loc.\ cit.\ only requires $\operatorname{Qcoh}(X)$ to be a locally Noetherian Grothendieck abelian category (see \cite[\S 3]{Krause:2005}). However,  when coupled with \cite[Lemma B.3]{Coupek/Stovicek:2020}, it follows that \cite[Theorem 1.1]{Krause:2005} is true for arbitrary Noetherian schemes. Furthermore, $K(\operatorname{Inj}(X))$ is compactly generated and $Q_\rho$ identifies $D^b_{\operatorname{coh}}(X)$ with the subcategory $K^c(\operatorname{Inj}(X))$ of compact objects of $K(\operatorname{Inj}(X))$, see \cite[Theorem 1.1(2), Remark 3.8]{Krause:2005}. Furthermore, $S_{\operatorname{qc}}(X)$ is also compactly generated $D_{\operatorname{sg}}(X)$ and $I_\lambda Q_\rho$ identifies the idempotent completion of the singularity category $D_{\operatorname{sg}}(X) = D^b_{\operatorname{coh}}(X)/\operatorname{Perf}(X)$ with $S^c_{\operatorname{qc}}(X)$, see \cite[Theorem 1.1(3)]{Krause:2005}. Note that the functor $I_\lambda Q_\rho$ identifies with the Verdier quotient functor $\pi\colon D_{\operatorname{coh}}^b(X) \to D_{\operatorname{sg}}(X)$.
\end{example}

Consider a recollement
\begin{displaymath}
    \begin{tikzcd}
        {\mathcal{T}} && {\mathcal{K}} && {\mathcal{D}.}
        \arrow["I"{description}, from=1-1, to=1-3]
        \arrow["{I_\lambda}"', bend right =30pt, from=1-3, to=1-1]
        \arrow["{I_\rho}", bend right =-30pt, from=1-3, to=1-1]
        \arrow["Q"{description}, from=1-3, to=1-5]
        \arrow["{Q_\lambda}"', bend right =30pt, from=1-5, to=1-3]
        \arrow["{Q_\rho}", bend right =-30pt, from=1-5, to=1-3]
    \end{tikzcd}
\end{displaymath}
Assume $(\mathcal{T}^{\leq 0}, \mathcal{T}^{\geq 0})$ and $(\mathcal{D}^{\leq 0}, \mathcal{D}^{\geq 0})$ are $t$-structures respectively on $\mathcal{T}$ and $\mathcal{D}$. It is possible to `glue' these aisles to obtain a new $t$-structure on $\mathcal{K}$. Particularly, the glued $t$-structure is given by 
\begin{displaymath}
    \begin{aligned}
        \mathcal{K}^{\leq 0} &\colonequals  \left\{ E\in \mathcal{K} \mid I_\lambda (E) \in \mathcal{T}^{\leq 0}, Q(E) \in \mathcal{D}^{\leq 0} \right\},
        \\ \mathcal{K}^{\geq 0} &\colonequals  \left\{ E\in \mathcal{K} \mid I_\rho (E) \in \mathcal{T}^{\geq 0}, Q(E) \in \mathcal{D}^{\geq 0} \right\}.
    \end{aligned}
\end{displaymath}
Moreover, a few general facts include:
\begin{itemize}
    \item $I_\lambda$ and $Q_\lambda$ are right $t$-exact
    \item $I$ and $Q$ are $t$-exact
    \item $I_\rho$ and $Q_\rho$ are left $t$-exact.
\end{itemize}
See \cite[1.4.10 \& 1.3.17.(iii)]{Beilinson/Berstein/Deligne/Gabber:2018} for details.

\begin{lemma}
    \label{lem:glueing_recollement}
    Consider a recollement of triangulated categories
    \begin{displaymath}
        \begin{tikzcd}
            {\mathcal{T}} && {\mathcal{K}} && {\mathcal{D}.}
            \arrow["I"{description}, from=1-1, to=1-3]
            \arrow["{I_\lambda}"', bend right =30pt, from=1-3, to=1-1]
            \arrow["{I_\rho}", bend right=-30pt, from=1-3, to=1-1]
            \arrow["Q"{description}, from=1-3, to=1-5]
            \arrow["{Q_\lambda}"', bend right=30pt, from=1-5, to=1-3]
            \arrow["{Q_\rho}", bend right=-30pt, from=1-5, to=1-3]
        \end{tikzcd}
    \end{displaymath}
    Let $\mathcal{A}$ be an aisle on $\mathcal{K}$. Then the following are equivalent:
    \begin{enumerate}
        \item $\mathcal{A}$ is the aisle on $\mathcal{K}$ obtained by glueing the aisles $\mathcal{A}_{\mathcal{T}} \colonequals  \overline{\langle I_\lambda (\mathcal{A}) \rangle}^{(-\infty,0]}$ on $\mathcal{T}$ and $\mathcal{A}_{\mathcal{D}} \colonequals  \overline{\langle Q (\mathcal{A}) \rangle}^{(-\infty,0]}$ on $\mathcal{D}$
        \item $(Q_\rho \circ Q )(\mathcal{A}^{\perp})\subseteq \mathcal{A}^{\perp}$
        \item $(Q_\lambda \circ Q )(\mathcal{A})\subseteq \mathcal{A}$.
    \end{enumerate}
\end{lemma}

\begin{proof}
    This is \cite[Proposition 1.4.12]{Beilinson/Berstein/Deligne/Gabber:2018}; see \cite[Theorem 2.8]{Psaroudakis/Vitoria:2018} for a relevant statement.
\end{proof}

\subsection{Tensor actions}
\label{sec:prelim_tensor}

We give a short reminder on tensor actions of triangulated categories. Those not familiar are encouraged to read \cite[\S 3]{Stevenson:2013}. While it is possible to work in larger generality, we refrain from doing so due to the scope of our paper. 

Let $(\mathcal{T},\otimes,\mathbf{1})$ be a rigidly compactly generated tensor triangulated category and $\mathcal{K}$
be a compactly generated triangulated category.
Consider an action $\odot\colon \mathcal{T}\times\mathcal{K}\to \mathcal{K}$ of
$\mathcal{T}$ on $\mathcal{K}$; see \cite[Definition 3.2]{Stevenson:2013}. In
particular, $\odot$ is exact and coproduct preserving in both variables.
By \cite[Lemma 4.6]{Stevenson:2013}, this action
restricts to an action $\odot\colon \mathcal{T}^c\times\mathcal{K}^c\to
\mathcal{K}^c$ of the compacts of $\mathcal{T}$ on those of $\mathcal{K}$. Fix a preaisle $\mathcal{P}^{\leq 0}\subseteq\mathcal{T}^c$ such that
$\mathbf{1}\in\mathcal{P}^{\leq 0}$ and
$\mathcal{P}^{\leq 0}\otimes\mathcal{P}^{\leq 0}\subseteq\mathcal{P}^{\leq 0}$.
We denote by $\mathcal{T}^{\leq 0}\colonequals \overline{\langle \mathcal{P}^{\leq 0}
\rangle}^{(-\infty, 0]}$ the aisle generated by $\mathcal{P}^{\leq 0}$.

\begin{definition}
	We say that a preaisle $\mathcal{U}\subseteq\mathcal{K}$ is a \textbf{$\odot$-preaisle}
	if we have $\mathcal{P}^{\leq 0}\odot\mathcal{U}\subseteq\mathcal{U}$.
\end{definition}

\begin{remark}
    Our definition of $\odot$-preaisle differs slightly from
    \cite[Definition 3.2]{Dubey/Sahoo:2023} (where they consider the action of
    $\mathcal{T}$ on itself). Specifically, we only require closure under tensoring with $\mathcal{P}^{\leq 0}$ rather than $\mathcal{T}^{\leq 0}$. The reason is that we want to consider also $\odot$-preaisles of certain subcategories of $\mathcal{K}$, such as $\mathcal{K}^c$, and usually we do not even  that
    $\mathcal{T}^{\leq 0}\odot\mathcal{K}^c\subseteq\mathcal{K}^c$. 
However, the two notions coincide for cocomplete preaisles, as shown below.
\end{remark}

\begin{lemma}
    \label{lemma:cocomplete-tensor-preaisle}
	The following are equivalent for a cocomplete preaisle
	$\mathcal{A}\subseteq\mathcal{K}$:
	\begin{enumerate}
		\item $\mathcal{A}$ is a $\odot$-preaisle, as defined above;
		\item $\mathcal{T}^{\leq 0}\odot \mathcal{A}\subseteq\mathcal{A}$.
	\end{enumerate}
\end{lemma}

\begin{proof}
	Since $\mathcal{P}^{\leq 0} \subseteq \mathcal{T}^{\leq 0}$, it follows that $\mathcal{T}^{\leq 0}\odot \mathcal{A}\subseteq\mathcal{A}$ implies $\mathcal{P}^{\leq 0}\odot \mathcal{A}\subseteq\mathcal{A}$. So, we prove the converse by \emph{d\'{e}vissage}. Consider the subcategory $\mathcal{T}^\prime \colonequals \{\,X\in\mathcal{T}\mid X\odot \mathcal{A}\subseteq\mathcal{A}\,\}$ of $\mathcal{T}$. Since $\odot$ is triangulated and coproduct preserving in the first variable and $\mathcal{A}$ is a cocomplete preaisle, $\mathcal{T}^\prime$ is a cocomplete preaisle as well. Moreover, $\mathcal{T}^\prime$ contains $\mathcal{P}^{\leq 0}$ by the assumption, and so $\mathcal{T}^\prime$ contains $\mathcal{T}^{\leq 0}$. This finishes the proof.
\end{proof}

\begin{definition}
	Let $\mathcal{S}\subseteq\mathcal{K}$ be a set of objects. We denote by 
	$\overline{\langle \mathcal{S} \rangle}^{(-\infty, 0]}_\odot$ the smallest
	cocomplete $\odot$-preaisle of $\mathcal{K}$ containing $\mathcal{S}$ (which
	is obtained by intersecting all of them).
\end{definition}

\begin{lemma}
    \label{lem:generated_odot_aisle}
	For any skeletally small subcategory $\mathcal{S}\subseteq\mathcal{K}$, we have:
	\begin{displaymath}
        \overline{\langle \mathcal{S} \rangle}^{(-\infty, 0]}_\odot =
		\overline{\langle \mathcal{P}^{\leq 0}\odot\mathcal{S} \rangle}^{(-\infty, 0]}.
    \end{displaymath}
	In particular,
	$\overline{\langle \mathcal{S} \rangle}^{(-\infty, 0]}_\odot$ is an aisle.
\end{lemma}

\begin{proof}
    First, we show $\overline{\langle \mathcal{P}^{\leq 0}\odot\mathcal{S} \rangle}^{(-\infty, 0]} \subseteq \overline{\langle \mathcal{S} \rangle}^{(-\infty, 0]}_\odot $. By construction, 
	$\overline{\langle \mathcal{S} \rangle}^{(-\infty, 0]}_\odot$ is a cocomplete
	preaisle containing $\mathcal{S}$ and closed under tensoring with
	$\mathcal{P}^{\leq 0}$. Therefore, it contains
	$\mathcal{P}^{\leq 0}\odot\mathcal{S}$, and so also the smallest cocomplete
	preaisle containing this subcategory. 
    
    Now, we check the reverse inclusion. Again by construction, 
	$\overline{\langle \mathcal{P}^{\leq 0}\odot\mathcal{S} \rangle}^{(-\infty,
	0]}$ is a cocomplete preaisle containing
	$\mathcal{S}=\mathbf{1}\odot\mathcal{S}\subseteq\mathcal{P}^{\leq 0}\odot\mathcal{S}$.
	We need to show that it is a $\odot$-preaisle, and we argue by
	\emph{d\'{e}vissage}. Consider the subcategory
	\begin{displaymath}
        \mathcal{K}^\prime \colonequals \left\{\,X\in\mathcal{K}\mid
	\mathcal{P}^{\leq 0}\odot X\in
		\overline{\langle \mathcal{P}^{\leq 0}\odot\mathcal{S} \rangle}^{(-\infty,
		0]}\,\right\}
    \end{displaymath}
	of $\mathcal{K}$. Since $\odot$ is exact and coproduct preserving in the
	second variable and $\overline{\langle \mathcal{P}^{\leq 0}\odot\mathcal{S}
	\rangle}^{(-\infty,0]}$ is a cocomplete preaisle, $\mathcal{K}^\prime$ is a
	cocomplete preaisle as well. Since
	\begin{displaymath}
        \mathcal{P}^{\leq 0}\odot (\mathcal{P}^{\leq 0}\odot\mathcal{S})=
	(\mathcal{P}^{\leq 0}\otimes\mathcal{P}^{\leq 0})\odot\mathcal{S}\subseteq
	\mathcal{P}^{\leq 0}\odot\mathcal{S}\subseteq\overline{\langle
	\mathcal{P}^{\leq 0}\odot\mathcal{S} \rangle}^{(-\infty,0]},
    \end{displaymath}
	$\mathcal{K}^\prime$ contains $\mathcal{P}^{\leq 0}\odot\mathcal{S}$, and therefore
	it contains $\overline{\langle \mathcal{P}^{\leq 0}\odot\mathcal{S} \rangle}^{(-\infty,
	0]}$ as well.

	For the last claim, since $\mathcal{P}^{\leq 0}$ is skeletally small, so is
	$\mathcal{P}^{\leq 0}\odot\mathcal{S}$ for any skeletally small subcategory
	$\mathcal{S}\subseteq\mathcal{K}$. Therefore, the generated cocomplete
	preaisle is an aisle.
\end{proof}

\begin{corollary}
    \label{cor:compactly_generated_odot_aisle}
	If $\mathcal{S}\subseteq\mathcal{K}^c$ is a skeletally small subcategory of
	compact objects, then so is $\mathcal{P}^{\leq 0}\odot\mathcal{S}$. Therefore,
	$\overline{\langle \mathcal{S} \rangle}^{(-\infty, 0]}_\odot=
		\overline{\langle \mathcal{P}^{\leq 0}\odot\mathcal{S} \rangle}^{(-\infty,
		0]}$ is compactly generated.
\end{corollary}

\begin{proof}
	Follows from the fact, recalled above, that $\odot$ restricts to an action
	of $\mathcal{T}^c$ on $\mathcal{K}^c$.
\end{proof}

\begin{proposition}
    \label{prop:compact_aisle_susp}
		Assume that $\mathcal{K}$ underlies a triangulated
		derivator and $\mathcal{P}\subseteq\mathcal{K}^c$ a set of compact objects.
		Then we have:
		\begin{displaymath}
            \overline{\langle \mathcal{P} \rangle}\cap\mathcal{K}^c=\langle\mathcal{P}\rangle
		    \qquad\text{and}\qquad
			\overline{\langle \mathcal{P}
			\rangle}^{(-\infty,0]}\cap\mathcal{K}^c=\langle\mathcal{P}\rangle^{(-\infty,0]}.
        \end{displaymath}
\end{proposition}

\begin{proof}
	It is easy to see that $\langle\mathcal{P}\rangle \subseteq \overline{\langle \mathcal{P} \rangle}\cap\mathcal{K}^c$ and $\langle\mathcal{P}\rangle^{(-\infty,0]} \subseteq \overline{\langle \mathcal{P} \rangle}^{(-\infty,0]}\cap\mathcal{K}^c$. We claim the reverse inclusions hold. In the first case, use \cite[Lemma 2.2]{Neeman:1992b}, whereas the second follows by combining \cite[Theorem 12.3]{Keller/Nicolas:2013}
	together with \cite[Proposition 3.12]{Rouquier:2008}.
\end{proof}

\begin{remark}
    \label{rmk:keller_nicholas}
	In \Cref{prop:compact_aisle_susp}, the assumption of $\mathcal{K}$ underlying a triangulated derivator is used in
	\cite[Theorem 12.3]{Keller/Nicolas:2013} in order to be able to argue using
	functorial homotopy filtered colimits. This is a very weak assumption
	satisfied whenever $\mathcal{K}$ is the homotopy category of a stable model
	category, see \cite[Example 1.16]{Cisinski/Neeman:2008}. In fact, this includes our setting. 
\end{remark}

\begin{proposition}
    \label{prop:suspended_aisles_to_compactly_generated_aisles_a_la_neeman}
    With $\mathcal{K}$ as in \Cref{prop:compact_aisle_susp}, there exists a one-to-one correspondence: 
    \begin{displaymath}
        \Phi \colon \left\{ \odot \textrm{-suspended subcategories of } \mathcal{K}^c \right\} \leftrightarrows \left\{ \textrm{compactly generated } \textrm{$\odot$-aisles on } \mathcal{K} \right\} \colon \Psi
    \end{displaymath}
    where $\Phi\colon \mathcal{S}\mapsto \overline{\langle \mathcal{S} \rangle}^{(-\infty,0]}_\odot$ and $\Psi \colon \mathcal{A} \mapsto \mathcal{A} \cap  \mathcal{K}^c$.
\end{proposition}

\begin{proof}
    To start, we show the mappings are well-defined. First, let $\mathcal{S}$ be an $\odot$-suspended subcategory of $\mathcal{K}^c$. Consider the $\odot$-aisle $\overline{\langle \mathcal{S} \rangle}^{(-\infty,0]}_\odot$ on $\mathcal{K}$. As $\mathcal{T}^{\leq 0}$ is compactly generated by $\mathcal{P}^{\leq 0}$, \Cref{lem:generated_odot_aisle} implies $\overline{\langle \mathcal{S} \rangle}^{(-\infty,0]}_\odot = \overline{\langle \mathcal{P}^{\leq 0}\odot \mathcal{S} \rangle}^{(-\infty,0]}$. However, $\mathcal{P}^{\leq 0}\odot \mathcal{S}\subseteq \mathcal{K}^c$, and so, $\Phi$ is well-defined. Next, let $\mathcal{A}$ be a compactly generated $\odot$-aisle on $\mathcal{K}$. Consider the subcategory $\mathcal{A}\cap \mathcal{K}^c$. Note that $\mathcal{A}$ and $\mathcal{K}^c$ are both closed under direct summands, nonnegative shifts, and extensions. Also, for any $K\in \mathcal{A}\cap \mathcal{K}^c$ and $P\in \mathcal{P}^{\leq 0}$, one has $P\odot K\in \mathcal{A}$, whereas $P\odot K$ being compact ensures it belongs to $\mathcal{A}\cap \mathcal{K}^c$. Hence, $\Psi$ is well-defined.

    We are left to prove the one-to-one correspondence. First, we show $\Psi \circ \Phi$ is the identity. Let $\mathcal{S}$ be a $\odot$-suspended subcategory of $\mathcal{K}^c$. Clearly, $\mathcal{S}\subseteq \overline{\langle \mathcal{S} \rangle}^{(-\infty,0]}_\odot$, and so $\mathcal{S}\subseteq (\Psi \circ \Phi) (\mathcal{S})$. The reverse containment follows from \Cref{prop:compact_aisle_susp}.

    Next, we show $\Phi \circ \Psi$ is the identity. Let $\mathcal{A}$ be a compactly generated $\odot$-aisle on $\mathcal{K}$. As $\mathcal{A}$ is compactly generated by $\mathcal{A} \cap \mathcal{K}^c$, we know that $\mathcal{A} = \overline{\langle \mathcal{A} \cap \mathcal{K}^c \rangle}^{(-\infty,0]}$. It suffices to check that $\mathcal{A} \cap \mathcal{K}^c$ is an $\odot$-suspended subcategory of $\mathcal{K}^c$. However, we know this to be the case already, so $(\Phi \circ \Psi)(\mathcal{A}) =\mathcal{A}$.
\end{proof}

\begin{remark}
    \label{rmk:suspended_aisles_to_compactly_generated_aisles_a_la_neeman}
    Assume the notation of \Cref{prop:suspended_aisles_to_compactly_generated_aisles_a_la_neeman}. Given a compactly generated $\odot$-aisle $\mathcal{A}$ on $\mathcal{K}$, there are potentially many distinct choices for collections of $\mathcal{K}^c$ which generate $\mathcal{A}$. Indeed, e.g.\ on affine $n$-space $\mathbb{A}^n_k$ over a field $k$, $D^{\leq 0}_{\operatorname{qc}}( \mathbb{A}^n_k)$ on $D_{\operatorname{qc}}(\mathbb{A}^n_k)$ is compactly generated by the various sets $\mathcal{P}_s \colonequals  \{\mathcal{O}_{\mathbb{A}^n_k}[i] \}^s_{i=0}$ where $s\geq  0$. As a heuristic, \Cref{prop:suspended_aisles_to_compactly_generated_aisles_a_la_neeman} says $\mathcal{A}\cap \mathcal{K}^c$ is a maximal generating set for any such aisle which is $\odot$-suspended.
\end{remark}

\begin{example}
    \label{ex:actions}
    Let $X$ be a Noetherian scheme. Then $D_{\operatorname{qc}}(X)$ acts on $K(\operatorname{Inj}(X))$, $S_{\operatorname{qc}}(X)$, and itself. To not burden the reader with technical details, we only highlight things and point the reader to the literature:
    \begin{itemize}
        \item  $D_{\operatorname{qc}}(X)$ acts on itself via the derived tensor product $\otimes^{\mathbf{L}}$; when viewed as a tensor action of $D_{\operatorname{qc}}(X)$, we write this as $\otimes$. 
        \item $D_{\operatorname{qc}}(X)$ acts on $K(\operatorname{Inj}(X))$, which we denote by $\odot$, via the ind-completion of the action $\otimes\colon \operatorname{Perf}(X) \times D^b_{\operatorname{coh}}(X) \to D^b_{\operatorname{coh}}(X)$. See \cite[\S 1.8]{BILMP:2023} for a more concrete description. It essentially follows from taking tensor product of complexes and using $K$-flat resolutions to represent objects in $D_{\operatorname{qc}}(X)$.
        \item $D_{\operatorname{qc}}(X)$ acts on $S_{\operatorname{qc}}(X)$, which we denote by $\star$, from restriction of its action on $K(\operatorname{Inj}(X))$. See \cite[\S 3]{Stevenson:2013} for details.
        \item The role of $\mathcal{P}^{\leq 0}$ is in this setting played by $\operatorname{Perf}^{\leq 0}(X) = \operatorname{Perf}(X) \cap D^{\leq 0}_{\operatorname{qc}}(X)$.
    \end{itemize}
    In such special cases, these notations for the mentioned actions above will be used throughout our work.
\end{example}

\subsection{Thomason filtrations}
\label{sec:prelim_Thomason}

We recall Thomason filtrations on a Noetherian scheme $X$. To start, let $\operatorname{Spcl}(X)$ be the collection of specialization closed subsets of $X$. A function $\phi \colon \mathbb{Z} \to \operatorname{Spcl}(X)$ is called a \textbf{Thomason filtration} on $X$ if $\phi(n+1)\subseteq \phi (n)$ for each $n\in \mathbb{Z}$. There is a bijective correspondence between the collection of compactly generated tensor $t$-structures on $D_{\operatorname{qc}}(X)$ and the collection of Thomason filtrations on a $X$. Initially, this was proven for Noetherian affine schemes \cite[Theorem 3.11]{AlonsoTarrio/JeremiasLopez/Saorin:2010}, and later extended to Noetherian schemes \cite[Theorem 4.11]{Dubey/Sahoo:2023}. However, it has been fully generalized to affine schemes in \cite[Theorem 5.6]{Hrbek:2020}, and recently extended to algebraic stacks \cite[Theorem 1.3]{Hrbek/Lank/Pizzirani:2025}. We briefly recall the assignments for the case of a Noetherian scheme $X$.
\begin{enumerate}
    \item Let $\phi$ be a Thomason filtration on $X$. The associated category $\mathcal{U}_\phi$ is a compactly generated $\otimes$-aisle on $D_{\operatorname{qc}}(X)$, for which we will denote the corresponding compactly generated tensor $t$-structure by $(\mathcal{U}_\phi, \mathcal{V}_\phi)$ and the truncation functors by $\tau_\phi ^{\leq 0}$ and $ \tau_\phi ^{\geq 1}$. Specifically, $\mathcal{U}_\phi$ is the strictly full subcategory of $D_{\operatorname{qc}}(X)$ consisting of objects $E$ such that $\operatorname{supp}(\mathcal{H}^i(E))$ is contained in $\phi (i)$ for all integers $i$. The associated aisle $\mathcal{U}_\phi$ is compactly generated by the following subcategory:
    \begin{displaymath}
        \bigcup_{n \in \mathbb{Z}} \big(\operatorname{Perf}(X) \cap D_{\operatorname{qc},~\phi (n)}^{\leq n}(X)\big).
    \end{displaymath}
    \item Let $\mathcal{U}$ be a $\otimes$-aisle on $D_{\operatorname{qc}}(X)$ that is compactly generated by a collection $\mathcal{S}$. The associated Thomason filtration, denoted by $\phi_{\mathcal{U}}$, is given by the following:
    \begin{displaymath}
        \phi_{\mathcal{U}}(i)= \bigcup_{P \in \mathcal{S}} \bigcup_{j\geq i} \operatorname{supp}(\mathcal{H}^j (P)).
    \end{displaymath}
\end{enumerate}
Let $\phi$ be a Thomason filtration on $X$. We might at times call the corresponding compactly generated tensor $t$-structure its \textbf{associated $t$-structure on $D_{\operatorname{qc}}(X)$}, and vice-versa for the $t$-structure corresponding to $\phi$. These notions extend in a similar vein for aisles associated to Thomason filtrations.

\section{\texorpdfstring{$\otimes$}{t}-aisles generated by coherent complexes}
\label{sec:classify_pseudocoherent_aisles}

This section studies $\otimes$-aisles in $D_{\operatorname{qc}}(X)$ that are generated by complexes from $D^b_{\operatorname{coh}}(X)$ on a Noetherian scheme $X$. Our main result generalizes \cite[Theorem 3.10]{AlonsoTarrio/JeremiasLopez/Saorin:2010} from the affine case; see \Cref{thm:ALS_global_full_classification}. While this is of independent interest, we use these results for further classifications in later sections. 

\begin{lemma}\label{lemma:t-exactness-of-pullback-and-pushforward}
	Let $X$ be a quasi-compact quasi-separated scheme. Suppose $f\colon
	U\to X$ be an open immersion from an affine scheme. Then there exists $N\geq
	0$ (which can be taken to be zero if $X$ is semi-separated), such that for any set
	$\mathcal{S}\subseteq D_{\operatorname{qc}}(X)$, we have that:
 	\begin{displaymath}
        \mathbf{L}f^\ast\left(\overline{\langle \mathcal{S}\rangle}_\otimes^{(-\infty,0]}\right)\subseteq
				\overline{\langle \mathbf{L}f^\ast\mathcal{S} \rangle}_\otimes^{(-\infty,0]}
					\qquad\text{and}\qquad
			\mathbf{R}f_\ast\left(\overline{\langle \mathbf{L}f^\ast\mathcal{S} \rangle}_\otimes^{(-\infty,0]}\right)[N]\subseteq
				\overline{\langle \mathcal{S}\rangle}_\otimes^{(-\infty,0]}.
    \end{displaymath}
\end{lemma}

\begin{proof}
	The functor $\mathbf{L}f^\ast$ is $t$-exact with respect to the standard
	$t$-structures, it preserves compacts because it has coproduct preserving
	right adjoint, and it commutes with the tensor product. Moreover, it is itself
	coproduct preserving, in addition to being triangulated. The
	first inclusion follows by d\'{e}vissage. For the second, recall that we have
	$\mathbf{R}f_\ast\mathbf{L}f^\ast=\mathcal{O}_U\otimes-$ where $\mathcal{O}_U := \mathbf{R}f_\ast\mathbf{L}f^\ast(\mathcal{O}_X)$. By
	\textrm{\cite[\href{https://stacks.math.columbia.edu/tag/073G}{Tag
			073G}]{StacksProject}}, we have that
	$\mathcal{O}_U$ lies in $D_{\operatorname{qc}}^{\leq N}(X)$ for some $N\geq 0$
	(for $N=0$ if $X$ is semi-separated, in which case $f$ is an affine morphism). 
	Therefore 
	$\mathbf{R}f_\ast\mathbf{L}f^\ast\mathcal{S}[N]=\mathcal{O}_U[N]\otimes\mathcal{S}$
	lies in $\overline{\langle \mathcal{S}\rangle}_\otimes^{(-\infty,0]}$ by the
	tensor compatibility of this aisle. The wanted inclusion then follows by
	d\'{e}vissage, once we notice that since $U$ is affine we have
	$\overline{\langle\mathbf{L}f^\ast\mathcal{S}\rangle}_\otimes^{(-\infty,0]}=
	\overline{\langle\mathbf{L}f^\ast\mathcal{S}\rangle}^{(-\infty,0]}$, and
	$\mathbf{R}f_\ast$ is triangulated and coproduct preserving.
\end{proof}

\begin{lemma}\label{lemma:mayer-vietoris}
	Let $X$ be a quasi-compact quasi-separated scheme. Then there exists a finite open
	affine cover $X=\bigcup_{i\in I}U_i$ and an integer $n\geq0$ such that for every object $A\in
	D_{\operatorname{qc}}(X)$ the object $A[n]$ lies in the preaisle $\langle A_{|U_i}\mid i\in I
	\rangle^{(-\infty,0]}$, where $A_{|U_i} = \mathbf{R}(f_i)_\ast\mathbf{L}(f_i)^\ast(A)$ and $f_i: U_i \to X$ is the inclusion.
\end{lemma}

\begin{proof}
	By quasi-compactness, there exists a finite open affine cover
	$X=\bigcup_{j=1}^mV_j$. We first prove the claim under the additional
	assumption that $X$ is semi-separated, so that the intersections between
	any number of the $V_j$'s are still affine.
	Let $I$ be the set of subsets of $\{1,\dots,m\}$, let
	$\{U_{\underline i}\colonequals \bigcap_{j\in \underline i}V_j\mid \underline i\in
	I\}$ be the set of these affine intersections, and let $n\colonequals m$.
	To show that this choice satisfies the claim, we argue by induction on $m$, where the base
	case $m=1$ is trivial. Consider the open subscheme
	$X'\colonequals \bigcup_{j=2}^mV_j\subseteq X$ and the object $A'\colonequals A_{|X'}\in
	D_{\operatorname{qc}}(X')\subseteq D_{\operatorname{qc}}(X)$. By induction,
	we have that $A'[m-1]$ lies in $\langle A'_{|U_{\underline i}}\mid \underline
	i\subseteq \{2,\dots,m\}\rangle^{(-\infty,0]}$. Now, we have a Mayer--Vietoris
	triangle $A\to A_{|V_1}\oplus A'\to A_{|V_1\cap(\bigcup_{j=2}^mV_j)}\to A[1]$.
	Using again the induction step, this time on the open subscheme
	$X''\colonequals V_1\cap(\bigcup_{j=2}^mV_j)=\bigcup_{j=2}^m(V_1\cap V_j)$ and the object
	$A_{|X''}$, we deduce that $A_{|X''}[m-1]$ lies in $\langle A_{|U_{\underline
	i}}\mid 1\in\underline i\subseteq\{1,\dots,m\}\rangle^{(-\infty,0]}$. From
	(a shift and rotation of) the Mayer--Vietoris triangle we then obtain the claim.
	
	Now we address the more general case in which $X$ is quasi-separated, but not
	necessarily semi-separated. Our previous argument does not work as is, because
	the intersections $U_{\underline i}$, despite being quasi-compact, need not be
	affine.

	We start another induction argument on $m$. If $m=1$, the claim is again
	trivial. For $m>1$, consider as before the open subschemes
	$X'\colonequals \bigcup_{j=2}^mV_j\subseteq X'$ and
	$X''\colonequals V_1\cap(\bigcup_{j=2}^mV_j)=\bigcup_{j=2}^m(V_1\cap V_j)$.
	The first one admits by construction a cover with $m-1$ affine opens, so we can
	appeal to the inductive step. This is
	not necessarily the case for $X''$, as the intersections $V_1\cap V_j$ do
	not need to be affine anymore. However, it will still be a quasi-compact
	scheme, and since it is a subscheme of the affine scheme $V_1$, it will also
	be semi-separated. Therefore we can apply to it the first part of this proof
	(with a different integer $m$, which will be the cardinality of an affine open
	cover of $X''$). This will yield affine open covers
	$\{U'_i\mid i\in I'\}$ of $X'$ and $\{U''_i\mid i\in I''\}$ of $X''$ and
	integers $n',n''\geq 0$ such that
	$A_{|X'}[n']$ lies in $\langle A_{|U'_i}\mid i\in I'\rangle^{(-\infty,0]}$ and
	$A_{|X''}[n'']$ lies in $\langle A_{|U''_i}\mid i\in
	I''\rangle^{(-\infty,0]}$, for every object $A\in D_{\operatorname{qc}}(X)$.
	From the Mayer--Vietoris triangle $A\to A_{|V_1}\oplus A_{|X'}\to A_{|X''}\to
	A[1]$, we deduce that for $n\colonequals \max\{n',n''+1\}$ we have that $A[n]$ lies in
	$\langle A_{V_1}, A_{|U'_i}, A_{|U''_i}\rangle^{(-\infty,0]}$.
\end{proof}

\begin{lemma}\label{lemma:standard-relative}
	Let $X$ be a Noetherian scheme. For any $\mathcal{S}\subseteq
	D^b_{\operatorname{coh}}(X)$ with $V\colonequals \operatorname{supp}\mathcal{S}$, there exists an
	integer $\bar n$ such that $D_{\operatorname{qc},V}^{\leq
	\bar n}(X)\subseteq\overline{\langle \mathcal{S}\rangle}_\otimes^{(-\infty,0]}$.
\end{lemma}

\begin{proof}
	Let $X=\bigcup_{i\in I}$ and $n\geq 0$ be the finite open affine cover and the
	integer given by Lemma~\ref{lemma:mayer-vietoris}. For any $i\in I$, denote by
	$f_i\colon U_i\to X$ the corresponding open
	immersion. Let $N_i\geq0$ be the integers given by
	\Cref{lemma:t-exactness-of-pullback-and-pushforward} applied to $f_i$.
	Since
	$\mathbf{L}f^\ast\mathcal{S}\subseteq D^b_{\operatorname{coh}}(U_i)$ consists
	of bounded complexes with coherent cohomologies over a Noetherian affine scheme $U_i$, it
	generates a compactly generated $t$-structure by \cite[Theorem
	3.10]{AlonsoTarrio/JeremiasLopez/Saorin:2010} (which is
	automatically a $\otimes$-$t$-structure). Now, by the classification of
	compactly generated $t$-structures over Noetherian affine schemes, there exist
	integers $n_i$ such that for every
	object $A\in D_{\operatorname{qc},V}(X)$ one has that
	$\mathbf{L}f_i^\ast(A)[n_i]$ lies in the aisle of the $t$-structure generate
	by $\mathbf{L}f_i^\ast\mathcal{S}$ (see \cite[Proposition 3.8 \& Theorem 3.10]{AlonsoTarrio/JeremiasLopez/Saorin:2010}). By
	\Cref{lemma:t-exactness-of-pullback-and-pushforward},
	we deduce that the objects
	$A_{|U_i}\colonequals\mathbf{R}(f_i)_\ast\mathbf{L}f_i^\ast(A)$ are such that
	$A_{|U_i}[N_i+\max\{n_i\mid
	i\in I\}]$ all lie in the $\otimes$-aisle generated by $\mathcal{S}$.
	By taking $\bar n\colonequals n+\max\{N_i\mid i \in I\}+\max\{n_i\mid i\in I\}$ we then obtain that $A[\bar
	n]$ lies in this aisle, by \Cref{lemma:mayer-vietoris}. This proves the
	claim.
\end{proof}

\begin{lemma}\label{lem:t_exactness_flat_pullback}
	Let $X$ be a Noetherian scheme. Suppose $f\colon U\to X$ is an open immersion
	from an affine scheme. For any $\mathcal{S}\subseteq
	D^b_{\operatorname{coh}}(X)$, one has that $\mathbf{L}f^\ast$ induces a
	$t$-exact functor:
	\[(D_{\operatorname{qc}}(X),\overline{\langle \mathcal{S}\rangle}_\otimes^{(-\infty,0]})
	\to
	(D_{\operatorname{qc}}(U),\overline{\langle
	\mathbf{L}f^\ast\mathcal{S}\rangle}_\otimes^{(-\infty,0]})\]
\end{lemma}

\begin{proof}
	Right $t$-exactness comes from
	\Cref{lemma:t-exactness-of-pullback-and-pushforward}.
	To prove left $t$-exactness, for any object $B$ of the coaisle
	$\left(\overline{\langle
	\mathcal{S}\rangle}_\otimes^{(-\infty,0]}\right)^\bot$, we need to prove that
	$\mathbf{L}f^\ast(B)$ lies in the coaisle
	$\left(\overline{\langle\mathbf{L}f^\ast\mathcal{S}\rangle}_\otimes^{(-\infty,0]}\right)^\bot$.
	Observe however that over the affine scheme $U$, every aisle is automatically a
	$\otimes$-aisle, since the compact objects of the standard aisle are the
	objects of $\operatorname{susp} \mathcal{O}_U$. This means that in
	$D_{\operatorname{qc}}(U)$ we have
	$\overline{\langle\mathbf{L}f^\ast\mathcal{S}\rangle}_\otimes^{(-\infty,0]}
	=\overline{\langle\mathbf{L}f^\ast\mathcal{S}\rangle}^{(-\infty,0]}$.

	Assume at first that $B$ is bounded below, $B\in D^+_{\operatorname{qc}}(X)$. Then, for
	every object $A\in\mathcal{S}$ and every $i\leq0$, we have:
	\begin{displaymath}
		\begin{aligned}
			\operatorname{Hom} & (\mathbf{L}f^\ast A , \mathbf{L}f^\ast B[i]) \cong
				\operatorname{Hom}(\mathcal{O}_U\otimes^{\mathbf{L}}A,\mathbf{L}f^\ast
				B[i]) \\
			&\cong \operatorname{Hom}(\mathcal{O}_U,
			\operatorname{\mathbf{R}\mathcal{H}\! \mathit{om}}(\mathbf{L}f^\ast A,
			\mathbf{L}f^\ast B[i]) ) &&
			(\textrm{\cite[\href{https://stacks.math.columbia.edu/tag/08DH}{Tag
			08DH}]{StacksProject}}) \\
			&\cong \operatorname{Hom}(\mathcal{O}_U, \mathbf{L}f^\ast
			\operatorname{\mathbf{R}\mathcal{H}\! \mathit{om}}(A, B[i]) ) &&
			(\textrm{\cite[ \href{https://stacks.math.columbia.edu/tag/0A6H}{Tags 0A6H}
			\& \href{https://stacks.math.columbia.edu/tag/08DL}{			08DL}]{StacksProject}}) \\
			&\cong H_{U}^0(\mathbf{L}f^\ast \operatorname{\mathbf{R}\mathcal{H}\!
			\mathit{om}}(A, B[i])) &&\textrm{($U$ is affine)} \\
			&\cong \mathbf{L}f^\ast(H_X^0(\operatorname{\mathbf{R}\mathcal{H}\!
			\mathit{om}}(A, B[i]))) &&\textrm{($\mathbf{L}f^\ast$ is $t$-exact for
			standard)} \\
			&\cong \mathbf{L}f^\ast(\operatorname{Hom}(A, B[i])=0.
			\end{aligned}
	\end{displaymath}
	This shows that $\mathbf{L}f^\ast B$ lies in the coaisle of the
	($\otimes$-)$t$-structure generated by $\mathbf{L}f^\ast\mathcal{S}$ in
	$D_{\operatorname{qc}}(U)$.
	We remark that in the second step this uses our assumption that $B$ is bounded
	below: in this case, the first reference provided shows that
	$\operatorname{\mathbf{R}\mathcal{H}\!\mathit{om}}(\mathbf{L}f^\ast
	A,\mathbf{L}f^\ast B)$ is computed in the same way in the categories
	$D_{\operatorname{qc}}(X)$ (where we are) and $D(\mathcal{O}_X)$ (where the
	formula provided by the second reference holds).

	Now, assume more generally that $B$ is an unbounded complex of the coaisle,
	and let $V\subseteq X$ be the support of the objects of $\mathcal{S}$.
	Consider the triangle $\Gamma_VB\to B\to L_VB\to \Gamma_VB[1]$, where
	$\Gamma_VB$ is supported on $V$ and $L_VB$ is supported outside of $V$.
	Then since both $B$ (by assumption) and $L_VB$ (by support
	considerations) lie in the coaisle of the $\otimes$-$t$-structure generated by
	$\mathcal{S}$, so does $\Gamma_VB$. Moreover, since
	$\mathbf{L}f^\ast$ preserves support, we have that
	$\mathbf{L}f^\ast(L_VB)$ necessarily lies in the coaisle of the $t$-structure generated by
	$\mathbf{L}f^\ast\mathcal{S}$, again by support. On the other hand, we claim that now
	$\Gamma_VB$ is bounded below. Indeed, by \Cref{lemma:standard-relative}
	there is an integer $\bar n$ such
	that $D_{\operatorname{qc},V}^{\leq n}(X)$ is contained in the $\otimes$-aisle generated by
	$\mathcal{S}$, and therefore $\Gamma_VB\in D_{\operatorname{qc},V}(X)$, which lies in the coaisle
	generated by $\mathcal{S}$, must be contained in
	$D_{\operatorname{qc},V}^{>\bar n}(X)\subseteq D^+_{\operatorname{qc}}(X)$. Then by our previous
	discussion we have that $\mathbf{L}f^\ast(\Gamma_VB)$ lies in the coaisle of
	the $t$-structure generated by $\mathbf{L}f^\ast\mathcal{S}$, and therefore so
	does $\mathbf{L}f^\ast B$ by the triangle $\mathbf{L}f^\ast(\Gamma_VB)\to
	\mathbf{L}f^\ast(B)\to \mathbf{L}f^\ast(L_VB)\to
	\mathbf{L}f^\ast(\Gamma_VB)[1]$.
\end{proof}

\begin{lemma}
    \label{lem:aisles_based_off_support}
    Let $X$ be a Noetherian scheme and $\mathcal{S}_1,\mathcal{S}_2\subseteq D^b_{\operatorname{coh}}(X)$. Then $\overline{\langle \mathcal{S}_1 \rangle}^{(-\infty,0]}_{\otimes}\subseteq \overline{\langle \mathcal{S}_2 \rangle}^{(-\infty,0]}_{\otimes}$ if one has for every $k\in\mathbb{Z}$,
    \begin{displaymath}
        \bigcup_{E\in \mathcal{S}_1} \operatorname{supp}(\mathcal{H}^k (E))\subseteq \bigcup_{\substack{E^\prime \in \mathcal{S}_2 \\ i\geq k}} \operatorname{supp}(\mathcal{H}^i (E^\prime)).
    \end{displaymath} 
\end{lemma}

\begin{proof}
    Choose $A\in \overline{\langle \mathcal{S}_1 \rangle}^{(-\infty,0]}_{\otimes}$. Consider the truncation triangle of $A$,
    \begin{displaymath}
        \tau^{\leq 0}_{\mathcal{S}_2} A \to A \to \tau^{\geq 1}_{\mathcal{S}_2} A \to (\tau^{\leq 0}_{\mathcal{S}_2} A)[1],
    \end{displaymath}
    with respect to $\overline{\langle \mathcal{S}_2 \rangle}^{(-\infty,0]}_{\otimes}$. Let $X=\cup^n_{i=1} U_i$ be an affine open cover with associated open immersions $s_i \colon U_i \to X$. From \Cref{lem:t_exactness_flat_pullback}, we know that $\mathbf{L} s_i^\ast$ induces $t$-exact functors
    \begin{displaymath}
        (D_{\operatorname{qc}}(X), \overline{\langle \mathcal{S}_2 \rangle}^{(-\infty,0]}_{\otimes} ) \to (D_{\operatorname{qc}}(U_i), \overline{\langle \mathbf{L} s_i^\ast \mathcal{S}_2 \rangle}^{(-\infty,0]}_{\otimes} ).
    \end{displaymath}
    It follows that we have the truncation triangle of $\mathbf{L} s_i^\ast A$,
    \begin{displaymath}
        \mathbf{L} s_i^\ast \tau^{\leq 0}_{\mathcal{S}_2} A \to \mathbf{L} s_i^\ast A \to \mathbf{L} s_i^\ast \tau^{\geq 1}_{\mathcal{S}_2} A \to \mathbf{L} s_i^\ast (\tau^{\leq 0}_{\mathcal{S}_2} A)[1],
    \end{displaymath}
    with respect to $\overline{\langle \mathbf{L} s_i^\ast \mathcal{S}_2 \rangle}^{(-\infty,0]}_{\otimes}$.
    Then our hypothesis, coupled with \cite[Proposition 5.2]{Hrbek:2020}, implies $\mathbf{L} s_i^\ast \tau^{\leq 0}_{\mathcal{S}_2} A\in \overline{\langle \mathbf{L} s_i^\ast \mathcal{S}_2 \rangle}^{(-\infty,0]}_{\otimes}$. In other words, $\mathbf{L} s_i^\ast \tau^{\geq 1}_{\mathcal{S}_2} A = 0$ for all $i$. Hence, $\tau^{\geq 1}_{\mathcal{S}_2} A= 0$, and so, $A\in \overline{\langle \mathcal{S}_2 \rangle}^{(-\infty,0]}_{\otimes}$.
\end{proof}



\begin{proposition}
    \label{prop:aisles_are_compactly_generated}
    Let $X$ be a Noetherian scheme. If $\mathcal{S}\subseteq D^b_{\operatorname{coh}}(X)$, then $\overline{\langle \mathcal{S} \rangle}^{(-\infty,0]}_{\otimes}$ is compactly generated.
\end{proposition}

\begin{proof}
    We can prove the claim for the $\otimes$-$t$-structure generated by a single object $E\in D^b_{\operatorname{coh}}(X)$. This will give us a family of compact generators of $\langle E\rangle_\otimes^{(-\infty,0]}$ for every object of $E\in\mathcal{S}$, and we will just take their union.
    
    For clarity, we assume up to shift that $E$ is concentrated in degrees $[0,b]\cap \mathbb{Z}$. For every $s\in[0,b]\cap\mathbb{Z}$, write the (closed) support of $\mathcal{H}^s(E)$ as a finite union of irreducible closed subsets $\operatorname{supp}(\mathcal{H}^s(E))=\bigcup_{k=1}^{m_s} Z_s^k$. Denote by $i_s^k\colon Z_s^k\to X$, where $s\in[0,b]\cap\mathbb{Z}$ and $1\leq k\leq m_s$, the associated closed immersion of the reduced closed subscheme structure on $Z_s^k$. Set $A\colonequals\bigoplus_{s=0}^b\bigoplus_{k=1}^{m_s}(i_s^k)_\ast\mathcal{O}_{Z_i^k}[-s]$.
     Clearly, the support of $\mathcal{H}^k (A)$ coincides with that of $\mathcal{H}^k (E)$ for all $k\in \mathbb{Z}$, and therefore 
    by \Cref{lem:aisles_based_off_support} it follows that $\overline{\langle A \rangle}^{(-\infty,0]}_{\otimes}= \overline{\langle E \rangle}^{(-\infty,0]}_{\otimes}$. Now, for each $(i_s^k)_\ast \mathcal{O}_{Z_s^k} [-s]$, we can find a $P_s^k \in \operatorname{Perf}_{Z_s^k}(X)$ and morphism $P_s^k \xrightarrow{\alpha_s^k}(i_s^k)_\ast \mathcal{O}_{Z_s^k} [-s]$ such that $\mathcal{H}^t (\alpha_s^k)$ is an isomorphism for $t>s$ and $\mathcal{H}^s (\alpha_s^k)$ is surjective (see e.g.\ \cite[\href{https://stacks.math.columbia.edu/tag/08EL}{Tag 08EL}]{StacksProject}). Hence, after applying \Cref{lem:aisles_based_off_support}, $\overline{\langle \bigoplus^b_{s=0}\bigoplus_{k=1}^{m_s} P_s^k \rangle}^{(-\infty,0]}_{\otimes} = \overline{\langle A \rangle}^{(-\infty,0]}_{\otimes}$. 
    Therefore, from \Cref{cor:compactly_generated_odot_aisle}, the desired claim follows.
\end{proof}

\begin{theorem}
    \label{thm:ALS_global_full_classification}
    Let $X$ be a Noetherian scheme. Then the following are equivalent for any $\otimes$-aisle $\mathcal{A}$ on $D_{\operatorname{qc}}(X)$:
    \begin{enumerate}
        \item \label{thm:ALS_global_full_classification1} $\mathcal{A}=\mathcal{U}_{\phi}$ associated to a Thomason filtration $\phi$ on $X$
        \item \label{thm:ALS_global_full_classification2} $\mathcal{A} = \overline{\langle\mathcal{P} \rangle}^{(-\infty,0]}_{\otimes}$ for some $\mathcal{P}\subseteq \operatorname{Perf}(X)$
        \item \label{thm:ALS_global_full_classification3} $\mathcal{A} = \overline{\langle \mathcal{B} \rangle}^{(-\infty,0]}_{\otimes}$ for some $\mathcal{B}\subseteq D^b_{\operatorname{coh}}(X)$.
    \end{enumerate}
    In particular, for each case above, there is $\mathcal{Q}\subseteq \operatorname{Perf}(X)$ such that $\mathcal{A} = \overline{\langle \mathcal{Q} \rangle}^{(-\infty,0]}$ (e.g.\ no tensor needed).
\end{theorem}

\begin{proof}
    $\eqref{thm:ALS_global_full_classification1} \iff \eqref{thm:ALS_global_full_classification2}$ is \cite[Theorem 1.3]{Hrbek/Lank/Pizzirani:2025};  
    $\eqref{thm:ALS_global_full_classification2} \implies \eqref{thm:ALS_global_full_classification3}$ is straightforward; $\eqref{thm:ALS_global_full_classification3} \implies \eqref{thm:ALS_global_full_classification2}$ is \Cref{prop:aisles_are_compactly_generated}.
\end{proof}

\section{\texorpdfstring{$t$}{}-proxy smallness}
\label{sec:proxy_smallness}

This section introduces a notion of proxy smallness for $t$-structures. Using
recollements, we characterize it in \Cref{thm:characterize_left_proxy_small}.  Let $\mathcal{T}$ be a triangulated category admitting coproducts, and denote by
$\mathcal{T}^c$ the subcategory of compacts.
Recall the following definition, which is originally due to \cite{Dwyer/Greenlees/Iyengar:2006b}, however we follow the formulation of \cite{Briggs/Iyengar/Stevenson:2025}.

\begin{definition}
	A skeletally small subcategory $\mathcal{S}\subseteq\mathcal{T}$ is \textbf{proxy small} if
	$\overline{\langle \mathcal{S} \rangle} = \overline{\langle \langle \mathcal{S} \rangle \cap
	\mathcal{T}^c \rangle}$. Equivalently, this means that $\overline{\langle
	\mathcal{S} \rangle}$ is compactly generated and $\overline{\langle
	\mathcal{S}\rangle} \cap \mathcal{T}^c \subseteq \langle\mathcal{S} \rangle$.
	In words, the objects of $\mathcal{S}$ finitely build the compacts in
	$\overline{\langle\mathcal{S}\rangle}\cap\mathcal{T}^c$, which in turn build
	back (not necessarily finitely) the objects of $\mathcal{S}$.
\end{definition}

To define $t$-proxy smallness, we mimick this definition, but we only allow
preaisle operations (that is, positive shifts, extensions and possibly
coproducts) to `build' objects.

\begin{definition}
	A skeletally small subcategory $\mathcal{S}\subseteq\mathcal{T}$ is
	\textbf{$t$-proxy small} if 
	\begin{displaymath}
        \overline{\langle \mathcal{S} \rangle}^{(-\infty,0]} =
	\overline{\langle \langle \mathcal{S} \rangle^{(-\infty,0]} \cap
	\mathcal{T}^c \rangle}^{(-\infty,0]}.
    \end{displaymath}
	That is, if the cocomplete preaisle generated by $\mathcal{S}$ is
	compactly generated, and its compact objects lie in the preaisle generated by $\mathcal{S}$. We say that an object is \textbf{$t$-proxy small} if it is as a singleton subcategory.
\end{definition}

\begin{remark} 
     Observe that a skeletally small subcategory $\mathcal{S}\subseteq\mathcal{T}$ is $t$-proxy small if and only if so is $\langle \mathcal{S}\rangle^{(-\infty,0]}$, and similarly it is proxy small if and only if so is $\langle \mathcal{S} \rangle$. Furthermore, an object $P \in \mathcal{T}$ is proxy small precisely when $\overline{\langle P \rangle}$ is compactly generated and $\overline{\langle P \rangle} \cap \mathcal{T}^c \subseteq \langle P \rangle$. Set $\mathbf{p} = \overline{\langle P \rangle} \cap \mathcal{T}^c$. In more humane terms, $P$ being proxy small amounts to
    \begin{itemize}
        \item $P$ finitely builds each object of $\mathbf{p}$, and
        \item $\mathbf{p}$ (not necessarily finitely) builds back the object $P$.
    \end{itemize}
    To leverage $t$-structures, we need to restrict these building procedures to only allow positive suspension. Thus, we might say that $P$ `finitely $t$-builds' $\mathbf{p}$ if $\mathbf{p} \subseteq \langle P \rangle^{(-\infty,0]}$ and that $\mathbf{p}$ $t$-builds $P$ if $P \in \overline{\langle \mathbf{p} \rangle}^{(-\infty,0]}$. Plugging these notions in leads precisely to our definition of $t$-proxy smallness above. A natural question occurs regarding the actual novelty of our definition. Assume that $\mathbf{p}$ `$t$-builds' $P$ and that $P$ finitely builds $\mathbf{p}$. Does it follow that $P$ `finitely $t$-builds' $\mathbf{p}$? We shall demonstrate in \Cref{ex:tilting_proxy_small} and \Cref{ex:proxy_small_semiArtinian} that the answer is \textit{no} in general. In particular, not every proxy small object is $t$-proxy small.
\end{remark}

We now show that $t$-proxy smallness is a refinement of proxy smallness. 

\begin{lemma}
    \label{lem:t-proxy_small_is_proxy_small}
	Let $\mathcal{T}$ be as in \Cref{prop:compact_aisle_susp}. Then a $t$-proxy-small
	subcategory of $\mathcal{T}$ is proxy-small.
\end{lemma}

\begin{proof}
	Let $\mathcal{S}\subseteq\mathcal{T}$ be a subcategory. Observe that
	$\mathcal{S}$ is ($t$-)proxy small if, and only if, the preaisle $\langle
	\mathcal{S} \rangle^{(-\infty,0]}$ is such. Hence, we may assume that
	$\mathcal{S}$ is a preaisle. We are therefore assuming that if we let
	$\mathcal{P}\colonequals\mathcal{S}\cap\mathcal{T}^c\subseteq\mathcal{S}$, we have
	$\overline{\langle \mathcal{S}\rangle}^{(-\infty,0]}=\overline{\langle
	\mathcal{P}\rangle }^{(-\infty,0]}$.
	Observe that
	\begin{displaymath}
        \overline{\langle \mathcal{S}\rangle} =
		\overline{\left\langle \overline{\langle \mathcal{S}
		\rangle}^{(-\infty,0]}\right\rangle} =
		\overline{\left\langle \overline{\langle \mathcal{P}
		\rangle}^{(-\infty,0]}\right\rangle} =
		\overline{\langle \mathcal{P}\rangle}.
    \end{displaymath}
	This shows that the smallest localizing subcategory containing $\mathcal{S}$
	is compactly generated. Now, by \Cref{prop:compact_aisle_susp}, we have 
	\begin{displaymath}
        \overline{\langle \mathcal{S}\rangle} \cap\mathcal{T}^c =
		\overline{\langle \mathcal{P}\rangle} \cap\mathcal{T}^c =
		\langle \mathcal{P} \rangle \subseteq \langle \mathcal{S} \rangle.
    \end{displaymath}
	This concludes the proof.
\end{proof}

\begin{remark}
    In \Cref{lem:t-proxy_small_is_proxy_small}, the converse implication does not hold. See \Cref{ex:tilting_proxy_small,ex:proxy_small_semiArtinian}.
\end{remark}

\subsection{Results}
\label{sec:proxy_smallness_results}

\begin{setup}
    \label{setup:BIS}
    Let $\mathcal{A}$ be an additive category whose homotopy category $K(\mathcal{A})$ is compactly generated. Fix $P\in K(\mathcal{A})$. Consider the object $E\colonequals\mathbf{R}\operatorname{Hom}_{K(\mathcal{A})} (P,P)$ in $K(\mathcal{A})$ (i.e.\ the `dg endomorphism algebra of $P$'). Set $D(E)$ to be the derived category of the dg endomorphism algebra $E$. Define $\mathbf{p}$ to be the strictly full subcategory $\overline{\langle P \rangle}^{(-\infty,0]}\cap K(\mathcal{A})^c \subseteq \langle P \rangle^{(-\infty,0]}$. Using dg enhancements, the ind-completion of $\langle \mathbf{p}\rangle$, denoted $D(\mathbf{p})$, can be identified with $\overline{\langle \mathbf{p} \rangle}$ in $K(\mathcal{A})$ with inclusion functor $i_\ast \colon \overline{\langle \mathbf{p} \rangle} \to K(\mathcal{A})$.
    There is a commutative diagram
    \begin{displaymath}
        \begin{tikzcd}
            {D(E)} \\
            {\overline{\langle \mathbf{p} \rangle}} & {K(\mathcal{A})} \\
            {D(\mathbf{p})}
            \arrow[shift left=3, from=1-1, to=2-1]
            \arrow[shift left=3, bend right = -13 pt, from=1-1, to=2-2]
            \arrow[from=2-1, to=1-1]
            \arrow["{i_\ast}", from=2-1, to=2-2]
            \arrow[shift left=3, from=2-1, to=3-1]
            \arrow[bend right = 13pt, from=2-2, to=1-1]
            \arrow["{i^!}", shift left=3, from=2-2, to=2-1]
            \arrow[shift left=3, bend right = -13 pt, from=2-2, to=3-1]
            \arrow["\cong", from=3-1, to=2-1]
            \arrow[bend right = 13pt, from=3-1, to=2-2]
        \end{tikzcd}
    \end{displaymath}
    whose edges are adjoint pairs of exact functors given by `tensor-hom adjunctions'. Specifically, the unlabelled \textit{left adjoints} are as follows
    \begin{itemize}
        \item $D(E)\to K(\mathcal{A})$ is $(-)\otimes_E P$ which factors through $\overline{\langle \mathbf{p} \rangle}$ by virtue of preserving small coproducts
        \item $D(\mathbf{p})\to \overline{\langle \mathbf{p} \rangle}$ is the triangulated equivalence determined by universal property of ind-completions (e.g.\ $\mathbf{R}\operatorname{Hom}_{\mathbf{p}} (-,p)\mapsto p$ for all $p\in \mathbf{p}$)
        \item $D(\mathbf{p})\to K(\mathcal{A})$ is the functor determined also by the universal property which we denote by $(-)\otimes_\mathbf{p} \mathbf{p}$.
    \end{itemize}
    See \cite[Eq.\ 2.1]{Briggs/Iyengar/Stevenson:2025} for details. To make this precise, we expand the diagram as follows:
    \begin{displaymath}
        \begin{tikzcd}
            &&& {K(\mathcal{A})} \\
            {D(E)} &&&&&& {D(\mathbf{p}).} \\
            &&& {\overline{\langle \mathbf{p} \rangle}}
            \arrow[""{name=0, anchor=center, inner sep=0}, "{\mathbf{R}\operatorname{Hom}_{K(\mathcal{A})} (P,-)}"'{pos=0.6}, bend right = 30pt , from=1-4, to=2-1]
            \arrow[""{name=1, anchor=center, inner sep=0}, "{\mathbf{R}\operatorname{Hom}_{K(\mathcal{A})} (\mathbf{p},-)}"{pos=0.7}, bend right =-30pt, from=1-4, to=2-7]
            \arrow["{i^!}", shift left=3, from=1-4, to=3-4]
            \arrow[""{name=2, anchor=center, inner sep=0}, "{(-)\otimes_E P}"', bend right =-18pt, from=2-1, to=1-4]
            \arrow[""{name=3, anchor=center, inner sep=0}, "{(-)\otimes_E P}", shift left=3, bend right =18pt, from=2-1, to=3-4]
            \arrow[""{name=4, anchor=center, inner sep=0}, "{(-)\otimes_\mathbf{p} \mathbf{p}}", bend right =18pt, from=2-7, to=1-4]
            \arrow[""{name=5, anchor=center, inner sep=0}, "\cong", bend right =-24pt, from=2-7, to=3-4]
            \arrow["{i_\ast = inc.}", from=3-4, to=1-4]
            \arrow[""{name=6, anchor=center, inner sep=0}, "{\mathbf{R}\operatorname{Hom}_{K(\mathcal{A})} (P,-)|_{rest.}}", bend right =-24pt, from=3-4, to=2-1]
            \arrow[""{name=7, anchor=center, inner sep=0}, shift left=3, bend right =18pt, from=3-4, to=2-7]
            \arrow["\dashv"{anchor=center, rotate=111}, draw=none, from=2, to=0]
            \arrow["\dashv"{anchor=center, rotate=-110}, draw=none, from=3, to=6]
            \arrow["\dashv"{anchor=center, rotate=69}, draw=none, from=4, to=1]
            \arrow["\dashv"{anchor=center, rotate=-70}, draw=none, from=7, to=5]
        \end{tikzcd}
    \end{displaymath}
    By \cite[Theorem 4.3]{Briggs/Iyengar/Stevenson:2025}, it follows that $P$ being proxy small is equivalent to $\overline{\langle P \rangle} = \overline{\langle\overline{\langle P \rangle}\cap K(\mathcal{A})^c \rangle}$ and the functor
    \begin{displaymath}
        \mathbf{R}\operatorname{Hom}_{K(\mathcal{A})} (P,-) \otimes_E P \colon K(\mathcal{A}) \to K(\mathcal{A})
    \end{displaymath}
    preserves small coproducts. In fact, if $K(\mathcal{A})$ satisfies the telescope conjecture (i.e.\ every smashing subcategory of K is generated by the compact objects it contains), then proxy smallness of $P$ is equivalent the functor
    \begin{displaymath}
        \mathbf{R}\operatorname{Hom}_{K(\mathcal{A})} (P,-) \otimes_E P \colon K(\mathcal{A}) \to K(\mathcal{A})
    \end{displaymath}
    preserving small coproducts. See \cite[Corollary 4.6]{Briggs/Iyengar/Stevenson:2025}. In particular, we get a recollement
    \begin{displaymath}
        \begin{tikzcd}
            {\ker((-)\otimes_E P)} & {D(E)} && {\overline{\langle \mathbf{p} \rangle}.}
            \arrow[from=1-1, to=1-2]
            \arrow[bend right = -24pt, from=1-2, to=1-1]
            \arrow[bend right =24pt, from=1-2, to=1-1]
            \arrow["{(-)\otimes_E P}", from=1-2, to=1-4]
            \arrow["{\mathbf{R}\operatorname{Hom}_{K(\mathcal{A})} (P,-)|_{rest.}}", bend right = -24pt, from=1-4, to=1-2]
            \arrow["\lambda"', bend right = 24pt, from=1-4, to=1-2]
        \end{tikzcd}
    \end{displaymath}
    Observe that $(-) \otimes_E P$ is an exact functor which preserves small products and coproduct because it is both a right and left adjoint. From Brown representability \cite[Theorem 5.1]{Neeman:1996}, there exists $Q\in D(E)^c$ such that $(-)\otimes_E P \cong \mathbf{R} \operatorname{Hom}_E (Q,-)$. Hence, $\lambda\cong (-)\otimes_E Q$; again, we abuse notation with these `tensor-hom adjunctions'.
\end{setup}

The reader should compare the following with \cite[Theorem 4.3]{Briggs/Iyengar/Stevenson:2025}.

\begin{theorem}
    \label{thm:characterize_left_proxy_small}
    With notation of \Cref{setup:BIS}. An object $P\in K(\mathcal{A})$ is $t$-proxy small if, and only if, the following are satisfied:
    \begin{enumerate}
        \item \label{thm:characterize_left_proxy_small1} $\overline{\langle P \rangle}^{(-\infty,0]}$ is compactly generated
        \item \label{thm:characterize_left_proxy_small2} $\mathbf{R}\operatorname{Hom}_{K(\mathcal{A})} (P,-) \otimes_E P \colon K(\mathcal{A}) \to K(\mathcal{A})$ preserves small coproducts
        \item \label{thm:characterize_left_proxy_small3} the standard $t$-structure of $D(E)$ glues along the recollement.
    \end{enumerate}
\end{theorem}

\begin{proof}
    To start, we show an object being $t$-proxy small implies conditions \eqref{thm:characterize_left_proxy_small1}, \eqref{thm:characterize_left_proxy_small2} and \eqref{thm:characterize_left_proxy_small3}. Clearly, $P$ being $t$-proxy small implies \eqref{thm:characterize_left_proxy_small1}. In fact, $t$-proxy smallness implies proxy smallness by \Cref{lem:t-proxy_small_is_proxy_small}, so we have \eqref{thm:characterize_left_proxy_small2} as well by \cite[Theorem 4.3]{Briggs/Iyengar/Stevenson:2025}. From $P$ being $t$-proxy small, we know that
    \begin{displaymath}
        \overline{\langle P \rangle}^{(-\infty,0]} \cap K(\mathcal{A})^c \subseteq \langle P \rangle^{(-\infty,0]}.
    \end{displaymath} 
    Applying the functor $\mathbf{R}\operatorname{Hom}_{K(\mathcal{A})}(P,-)$, which agrees with $\lambda$ on $\mathbf{p}$, we have a string of inclusions
    \begin{displaymath}
        \begin{aligned}
            \lambda (\overline{\langle P \rangle}^{(-\infty,0]} \cap K(\mathcal{A})^c) 
            &= \mathbf{R} \operatorname{Hom}_{K(\mathcal{A})} (P, \overline{\langle P \rangle}^{(-\infty,0]} \cap K(\mathcal{A})^c)
            \\&\subseteq \langle \mathbf{R} \operatorname{Hom}_{K(\mathcal{A})} (P,P) \rangle^{(-\infty,0]} 
            \\&\cong \langle E \rangle^{(-\infty,0]} 
            \\&\subseteq \overline{\langle E \rangle}^{(-\infty,0]} && (\textrm{in } D(E)).
        \end{aligned}
    \end{displaymath}
    Now, consider the following string of inclusions,
    \begin{displaymath}
        \begin{aligned}
            (\lambda \circ ((-)\otimes_E P)) (\overline{\langle E \rangle}^{(-\infty,0]}) 
            &\subseteq \lambda (\overline{\langle E \otimes_E  P \rangle}^{(-\infty,0]})
            \\&\subseteq \lambda (\overline{\langle P \rangle}^{(-\infty,0]})
            \\&\subseteq \lambda (\overline{\langle \overline{\langle P \rangle}^{(-\infty,0]} \cap K(\mathcal{A})^c \rangle}^{(-\infty,0]})
            \\&\subseteq \overline{\langle \lambda (\overline{\langle P \rangle}^{(-\infty,0]} \cap K(\mathcal{A})^c) \rangle}^{(-\infty,0]}
            \\&\subseteq \overline{\langle E \rangle}^{(-\infty,0]}.
        \end{aligned}
    \end{displaymath}
    Tying things together, we see that the standard aisle of $D(E)$ is closed under the (exact) endofunctor $\lambda \circ ((-)\otimes_E P)$ on $D(E)$. Hence, the standard $t$-structure glues along the recollement as desired (see \Cref{lem:glueing_recollement}).

    Lastly, we show the converse direction. Observe that \eqref{thm:characterize_left_proxy_small1} and \eqref{thm:characterize_left_proxy_small2} ensure that $P$ is proxy small. Hence, we know that $\overline{\langle P \rangle}^{(-\infty,0]} \cap K(\mathcal{A})^c  \subseteq \langle P \rangle$. Our goal is to show that $\overline{\langle P \rangle}^{(-\infty,0]} \cap K(\mathcal{A})^c  \subseteq \langle P \rangle^{(-\infty,0]}$. From \eqref{thm:characterize_left_proxy_small3}, we see that $\lambda(P) = \lambda (E\otimes_E P)\in \overline{\langle E \rangle}^{(-\infty,0]}$, which gives us 
    \begin{displaymath}
        \begin{aligned}
            \lambda (\overline{\langle P \rangle}^{(-\infty,0]} \cap K(\mathcal{A})^c) 
            &\subseteq \lambda (\overline{\langle P \rangle}^{(-\infty,0]})
            \\&\subseteq \overline{\langle \lambda (P) \rangle}^{(-\infty,0]}
            \\&\subseteq \overline{\langle E \rangle}^{(-\infty,0]}.
        \end{aligned}
    \end{displaymath}
    Since $\lambda (\overline{\langle P \rangle}^{(-\infty,0]} \cap K(\mathcal{A})^c)$ and $E$ are contained in $D(E)^c$, \Cref{prop:compact_aisle_susp} implies that $\lambda (\overline{\langle P \rangle}^{(-\infty,0]} \cap K(\mathcal{A})^c) \subseteq \langle E \rangle^{(-\infty,0]}$. Therefore, applying the exact functor $(-)\otimes_E P$, we have that
    \begin{displaymath}
        \begin{aligned}
            \overline{\langle P \rangle}^{(-\infty,0]} \cap K(\mathcal{A})^c
            &\cong (((-)\otimes_E P) \circ \lambda) (\overline{\langle P \rangle}^{(-\infty,0]} \cap K(\mathcal{A})^c)
            \\&\subseteq \langle E\otimes_E P \rangle^{(-\infty,0]} 
            \\&\subseteq \langle P \rangle^{(-\infty,0]}.
        \end{aligned}
    \end{displaymath}
    This completes the proof.
\end{proof}

\subsection{Examples}
\label{sec:proxy_smallness_example}

In what follows, we provide examples which illuminate the difference between proxy smallness and $t$-proxy smallness. We start with an example of an object in $D^b_{\operatorname{coh}}(R)$ which is proxy-small and fails to be $t$-proxy-small. As per \Cref{prop:LCI_characterizations}, a ring $R$ which provides such an example cannot be locally a complete intersection. 

For an object $X$ in a triangulated category $\mathcal{T}$, let us denote by $\langle X \rangle^{[0,\infty)}$ the smallest strictly full subcategory of $\mathcal{T}$ containing $X$ and closed under extensions, direct summands, and negative suspensions. 

This first lemma is trivial, and can be easily generalized, but it is all we will need to provide our first example. 

\begin{lemma}
    \label{lem:aisle_susp_coaisle_cosusp}
    Let $\mathcal{T}$ be a triangulated category with a fixed $t$-structure $(\mathcal{T}^{\leq 0}, \mathcal{T}^{\geq 0})$. Let $F$ be an exact (triangulated) endofunctor on $\mathcal{T}$.
    
    Suppose $X \in \mathcal{T}$.
        \begin{itemize}
            \item If $F(X) \in \mathcal{T}^{\leq i}$ for some $i$, then $F(Y) \in \mathcal{T}^{\leq i}$ for every $Y \in \langle X \rangle^{(-\infty,0]}$,
            \item If $F(X) \in \mathcal{T}^{\geq i}$ for some $i$, then $F(Y) \in \mathcal{T}^{\geq i}$ for every $Y \in \langle X \rangle^{[0,\infty)}$.
        \end{itemize}
\end{lemma}

\begin{proof}
    If $Y \in \langle X \rangle^{(-\infty,0]}$, then as $F$ is exact, $F(Y)\in \langle F(X) \rangle^{(-\infty,0]}$. As aisles are closed under nonnegative shifts, extensions, finite coproducts and direct summands, the conclusion follows. 
    The second point holds with an identical proof. 
\end{proof}

The following corollary is explicitly used to construct the desired counterexample. 

\begin{corollary}
    \label{cor:t-proxy-small-ring-homomorphism}
    Let $R\to S$ be a homomorphism of rings. Suppose $\mathbf{R} \operatorname{Hom}_{R}(X,Y) \in D^{\leq i}_{\operatorname{qc}}(R)$ for an $R$-module $X$, an $S$-module $Y$ and some integer $i$. If $Z_S \in \langle Y_S \rangle^{(-\infty,0]}$, where the suspension closure is considered in $D_{\operatorname{qc}}(S)$, then $\mathbf{R} \operatorname{Hom}_{R}(X,Z) \in D^{\leq i}_{\operatorname{qc}}(R)$.
\end{corollary}

\begin{proof}
    First, note that if $Z_S \in \langle Y_S \rangle^{(-\infty,0]}$, where the suspension closure is considered in $D_{\operatorname{qc}}(S)$, then the same statement holds when $Z$ and $Y$ are seen as complexes of $R$-modules. See also \cite[Theorem 3.13]{Dwyer/Greenlees/Iyengar:2006a}. The rest of the statement follows from \Cref{lem:aisle_susp_coaisle_cosusp}.
\end{proof}

\begin{remark}
    \Cref{lem:aisle_susp_coaisle_cosusp} and \Cref{cor:t-proxy-small-ring-homomorphism} can be seen as the suspension closure or suspension closure analogue to work in \cite{Dwyer/Greenlees/Iyengar:2006a}. In particular, finite homological dimensions are used in \cite[Section 5]{Dwyer/Greenlees/Iyengar:2006a} to give necessary conditions for an object to be virtually, or proxy small, guided in particular by Principle (5.2) in loc.\ cit. In our case, we need an invariant which is sensitive to either nonpositive or nonnegative shifts, thus inclusion in aisles or coaisles of $t$-structures accounts for this nonsymmetry. 
\end{remark} 

\begin{example}
    \label{ex:proxy-small-not-t-proxy-small}
    In the following, we construct a proxy small object in $D_{\operatorname{coh}}^b(R)$ which is not $t$-proxy small. The trivial ring extension construction is inspired by \cite[Theorem 5.5]{Dwyer/Greenlees/Iyengar:2006a} which was used to find an object which is not proxy small, see in particular \cite[Example 5.7]{Dwyer/Greenlees/Iyengar:2006a}. We first describe the module-theoretic setup we require.

    Let $R$ be a commutative Noetherian ring such that there exist an $R$-module $X$ and a finitely generated module $Y$ such that $\operatorname{Ext}^i_R(X, R)=0$ for all $i>0$, $\operatorname{Ext}^1_R(X,Y)\neq 0$ but $\operatorname{Ext}^i_R(X,Y)=0$ for all $i>1$. Such a collection of modules can be found over a local Artinian Gorenstein ring in \cite[Corollary 3.3(1)]{Jorgensen/Sega:2004}. More explicitly, using the notation in loc. cit., we let $q=1$, $R=A$, $X = M$, and $Y= T_q=T_1$. That $\operatorname{Ext}^i_R(X, R)=0$ for all $i>0$ follows as $R$ is Gorenstein of dimension zero.

    Let $R \to R \ltimes Y$ be the trivial ring extension. Then $(R \ltimes Y)[1] \oplus R$ is proxy small in $D_{\operatorname{qc}} (R\ltimes Y)$ as $(R \ltimes Y)[1]$ is compact in $D_{\operatorname{qc}} (R\ltimes Y)$, and the associated primes of the $R \ltimes Y$-modules $R \ltimes Y$ and $R$ coincide. We now claim that $(R \ltimes Y)[1] \oplus R$ is not $t$-proxy small, which amounts to showing that $R \ltimes Y \notin \langle (R \ltimes Y)[1] \oplus R \rangle^{(-\infty,0]}$. 

    By the setup, $\mathbf{R} \operatorname{Hom}_{R}(X,R \oplus (R \oplus Y)[1]))\in D^{\leq 0}_{\operatorname{qc}}(R)$. However, $H^1(\mathbf{R} \operatorname{Hom}_{R}(X,R\oplus Y )) \cong \operatorname{Ext}^1_R(X,R\oplus Y ) \neq 0$, so by \Cref{cor:t-proxy-small-ring-homomorphism},  $R \ltimes Y \notin \langle (R \ltimes Y)[1] \oplus R \rangle^{(-\infty,0]}$ as desired.
\end{example}

\begin{remark}
    In the previous example, we constructed a proxy small object which is not $t$-proxy small in the derived category of a local Artinian ring which is \emph{not} Gorenstein. To see this, note that if $R$ is zero dimensional Gorenstein, then if the trivial ring extension $R\ltimes Y$ is Gorenstein then $0=\operatorname{Ext}^1_{R\ltimes Y}(k, R\ltimes Y)$, which implies that $\operatorname{Ext}^1_R(k, Y)=0$ so necessarily $Y$ is  injective in $\operatorname{Mod}(R)$, which is not the case.  

    Given the previous argument, it seems to follow that using the trivial ring extension as we used in the example will not provide an example of a proxy small but not $t$-proxy small objective in the bounded derived category of a Gorenstein ring. 
\end{remark}

The ring used to produce \Cref{ex:proxy-small-not-t-proxy-small} has an interesting history, most significantly to demonstrate unexpected behavior of modules over Artinian Gorenstein rings outside of the complete intersection case. 

\begin{remark}
    The local Artinian Gorenstein ring in \Cref{ex:proxy-small-not-t-proxy-small} was used by Gasharov and Peeva in their work \cite{Gasharov/Peeva:1990} in connection with the following. Eisenbud conjectured that over a commutative local Noetherian ring, any finitely generated module with bounded Betti numbers has an eventually periodic minimal projective resolution of period two, which was shown to hold over local complete intersections. However, in the aforementioned work of Gasharov and Peeva, a counterexample is constructed, and they give an example of a non-periodic totally acyclic complex. 

    Later, Jorgensen and \c Sega in \cite{Jorgensen/Sega:2004} use the same ring to provide a counterexample to a conjecture of Auslander on the vanishing of $\operatorname{Ext}^i$ over an Artinian ring, namely, that if for a finitely generated $R$-module $M$, there exists an integer $n_M$ such that if $\operatorname{Ext}^i_R(M,N)=0$ for a finitely generated $R$-module $N$ and $i\gg0$, then $\operatorname{Ext}^i_R(M,N)=0$ for all $i\geq n_M$. This conjecture was already known to hold over local complete intersections, and a counterexample was known to necessarily require an object of infinite complete intersection dimension. In fact, if an object of $D_{\operatorname{coh}}^b(R)$ has finite complete intersection dimension, then it is proxy small, see \cite[Proposition 5.7]{Letz:2021}, \cite[Corollary 3.3]{Bergh:2009}. In fact, it follows that complexes of finite complete intersection dimension are $t$-proxy small, also by the proof of Bergh in loc. cit.. 
\end{remark}

Interestingly, it is more straightforward to construct a counterexample to the dual problem of the existence of a proxy small object in $D_{\operatorname{coh}}^b(R)$ which does not contain all the compacts with the same support in each cohomological degree in its cosuspension closure. 

\begin{example}
    \label{ex:proxy-small-not-co-t-proxy-small}
    In the following, we construct an example of a proxy small complex $P \in D_{\operatorname{coh}}^b(R)$ which has nonzero cohomology in cohomological degree zero but with no compact object with nonzero cohomology in degree zero in its cosuspension closure $\langle P \rangle^{[0,\infty)}$. 

    Let $(R, \mathfrak{m}, k)$ be a Noetherian local ring with $M$ a finitely generated $R$-module of projective dimension exactly one. Let $R \to R \ltimes M$ be the trivial ring extension. Then $P = (R \ltimes M)[-1] \oplus R$ is proxy small in $D_{\operatorname{qc}} (R\ltimes M)$ as $(R \ltimes M)[-1]$ is, and the associated primes of the $R \ltimes M$-modules $R \ltimes M$ and $R$ coincide. We now claim that $R \ltimes M \notin \langle (R \ltimes M)[-1] \oplus R \rangle^{[0,-\infty)}$ inside $D_{\operatorname{qc}} (R\ltimes M)$. 

    As seen in \cite[3.13]{Dwyer/Greenlees/Iyengar:2006a}, if $R \ltimes M \in \langle (R \ltimes M)[-1] \oplus R \rangle^{[0,-\infty)}$ in $D_{\operatorname{qc}} (R\ltimes M)$, then the same also holds in $D_{\operatorname{qc}} (R)$, by considering the the underlying $R$-module structure. In $D_{\operatorname{qc}} (R)$, the formulation simplifies to $R \oplus M \in \langle \Sigma^{-1}R \oplus \Sigma^{-1}M \oplus R \rangle^{[0,-\infty)}$. It follows that $(R[-1] \oplus M[-1] \oplus R)\otimes^\mathbf{L}_R k$ has vanishing cohomology in nonpositive degrees, however, $M\otimes^\mathbf{L}_R k$ has non-vanishing cohomology in cohomological degree $-1$, a contradiction
    by \Cref{lem:aisle_susp_coaisle_cosusp}. Note that over a local ring, if $X$ is a finitely generated module considered as a complex in cohomological degree zero, $\inf(X\otimes^\mathbf{L}_R k)$ is exactly the flat dimension of $X$.
\end{example}

\begin{remark}
    There are various aspects of the module set up in \Cref{ex:proxy-small-not-t-proxy-small} which make it a more delicate problem than \Cref{ex:proxy-small-not-co-t-proxy-small}.

    For instance, the dual setup to \Cref{ex:proxy-small-not-co-t-proxy-small} requires finding a module that is left $\operatorname{Ext}$-orthogonal to $R$, while every module is right $\operatorname{Ext}$-orthogonal to $R$, so  we have this for free in \Cref{ex:proxy-small-not-co-t-proxy-small}. Moreover, there are no finitely generated injective modules unless the ring is Artinian Gorenstein, and outside the Artinian case, being finitely generated is not inherited by cosyzygies. 
    Thus we are restricted as we require that $Y$ is finitely generated in order for the trivial ring extension to be Noetherian.

    However, to find an example as in \Cref{ex:proxy-small-not-co-t-proxy-small} in the local Artinian case, using the trivial ring extension and module set up of \cite[Corollary 3.3(1)]{Jorgensen/Sega:2004} and a dual version of \Cref{cor:t-proxy-small-ring-homomorphism} should suffice.
\end{remark}

\begin{example}
    \label{ex:dualizing_module}
    Let $R$ be a commutative Noetherian ring with a dualizing module $\omega$. Then $\omega$ is proxy-small if, and only if, $R$ is Gorenstein. Indeed, if $R$ is Gorenstein then $\omega \cong R$ is compact. On the other hand, if $\omega$ is proxy-small then $R \in \operatorname{thick}(\omega)$ as $\omega$ has full support. Since $\omega$ is of finite injective dimension, this forces $R$ to be of finite injective dimension, and thus $R$ to be Gorenstein. In fact, the same argument shows that $R$ is Gorenstein if and only if $\omega$ is $t$-proxy-small.

    Outisde of the Gorenstein case, it is not clear whether $\omega$ is a direct summand in a $t$-proxy small object in a nontrivial way, that is, we consider the following. 

    Set $P = \omega \oplus R[1]$. Then $P$ is proxy-small and $\overline{\langle P \rangle}^{(-\infty,0]} = D^{\leq 0} = \overline{\langle R \rangle}^{(-\infty,0]}$, showing that $P$ satisfies condition \eqref{thm:characterize_left_proxy_small1} of \Cref{thm:characterize_left_proxy_small}. Is there $R$ such that $P$ is not $t$-proxy-small?

    If $(R,\mathfrak{m},k)$ is a local Artinian ring then the role of $\omega$ is played by the injective envelope $E(k)$ of $k$. Let $R$ be such that $\mathfrak{m}^2 = 0$ but non-Gorenstein. Consider the short exact sequence $0 \to R \to E(R) \to C \to 0$, where $0 \to R \to E(R)$ is the essential embedding of $R$ to its injective envelope. By the Matlis' structure theory of injectives, we have that $E(R) \cong \omega^n$ for some $n>0$. Since $R \to E(R)$ is essential and $\mathfrak{m}^2 = 0$, we also have $C \cong k^m$ for some $m > 0$. The induced rotated triangle $E(R) \to C \to R[1] \to$ shows that $C \in \langle P \rangle^{(-\infty,0]}$, which in turn yields $k \in \langle P \rangle^{(-\infty,0]}$. Since $R$ is of finite length, we conclude that $R \in \langle P \rangle^{(-\infty,0]}$ and $P$ is $t$-proxy-small.
\end{example}

In the rest of this section, we provide some examples outside of $D^b_{\operatorname{coh}}(R)$. 

\begin{example}
   \label{ex:tilting_proxy_small}
    An object $T \in K(\mathcal{A})$ is called \textbf{silting} provided that $\overline{\langle T \rangle}^{(-\infty,0]}$ coincides with the orthogonal subcategory
    \begin{displaymath}
        T^{\perp_{>0}} = \{X \in K(\mathcal{A}) \mid \operatorname{Hom}_{K(\mathcal{A})}(T,X[i]) = 0 ~\forall i>0\}.
    \end{displaymath}
    As a consequence of the definition, $T$ is a generator of $K(\mathcal{A})$ such that $T \in T^{\perp_{>0}}$, for more details see e.g. \cite[\S 4.2]{Angeleri:2019}. Although $T$ may fail to be proxy-small in general, it is always \textbf{equivalent} to a proxy-small one, in the sense that there is a set $\kappa$ such that the coproduct $T^{(\kappa)}$ of $\kappa$ copies of $T$ is proxy-small, see \cite[Lemma 4.2]{Hrbek:2024}. As $T \in T^{\perp_{>0}}$, the endomorphism dg-algebra $E$ of $T$ is connective. This implies that the standard aisle $\overline{\langle E \rangle}^{(-\infty,0]}$ coincides with $E^{\perp_{>0}}$, in other words, $E$ itself is a silting object of $D(E)$. The adjunction isomorphism 
    \begin{displaymath}
        \operatorname{Hom}_{D(E)}(E[i],\mathbf{R}\operatorname{Hom}_{K(\mathcal{A})}(T,-)) \cong \operatorname{Hom}_{K(\mathcal{A})}(T[i],-)
    \end{displaymath}
    shows that $\mathbf{R}\operatorname{Hom}_{K(\mathcal{A})}(T,-): K(\mathcal{A}) \to D(E)$ is $t$-exact, implying that the condition \eqref{thm:characterize_left_proxy_small3} of \Cref{thm:characterize_left_proxy_small} holds. Indeed, this follows from the characterization of gluing of \Cref{lem:glueing_recollement} and the discussion which precedes it---the functor $Q = (- \otimes_E T)$ is left $t$-exact, so recalling that $Q_\rho = \mathbf{R}\operatorname{Hom}_{K(\mathcal{A})}(T,-)$ we see that $Q_\rho Q$ is left $t$-exact, as required.
    
    There is however a supply of non-compact silting objects $T$ for which the condition \eqref{thm:characterize_left_proxy_small1} of \Cref{thm:characterize_left_proxy_small} fails. Indeed, if the silting $t$-structure $(T^{\perp_{>0}},T^{\perp_{\leq 0}})$ is compactly generated, its heart is a Grothendieck category \cite[Theorem C]{Saorin/Stovicek/Virili:2023} and has a projective generator \cite[Proposition 4.3]{Psaroudakis/Vitoria/2018}. Then \cite[Theorem 3.7]{Angeleri/Marks/Vitoria:2017} shows that $T$ is a pure-projective object of $K(\mathcal{A})$. Pure-projective silting objects which are not equivalent to a compact one exist, but are difficult to construct \cite{Bazzoni/Herzog/Prihoda/Saroch/Trlifaj:2020}. In particular, let $K(\mathcal{A}) = D_{\operatorname{qc}}(R)$ for a commutative ring $R$ and assume that $T$ is a tilting object represented by an $R$-module of projective dimension at most one. Then \cite[Theorem 3.7]{Bazzoni/Herzog/Prihoda/Saroch/Trlifaj:2020} asserts that $T$ is pure-projective if, and only if, it is projective, and then it is equivalent to a compact silting object. As an explicit example of a proxy-small silting object failing the condition \eqref{thm:characterize_left_proxy_small1}, we can consider $R = \mathbb{Z}$ and $T = \mathbb{Q} \oplus \mathbb{Q}/\mathbb{Z}$, or more generally see \cite[Example 8.4]{Positselski/Stovicek:2021}.
\end{example}   

\begin{example}
    \label{ex:proxy_small_semiArtinian}
    Let $R$ be a local semi-Artinian ring which is not Artinian. This is to say, $R$ is a commutative ring admitting a unique simple module $k$ up to isomorphism such that $R$ has composition series, but its Loewy length is infinite. For example, the following ring fits the description:
    \begin{displaymath}
        R = \left\{ \begin{pmatrix} a & b \\ 0 & a \end{pmatrix} \mid a \in \mathbb{Q}, b \in \mathbf{R}\right\}.
    \end{displaymath}
    Set $P \colonequals k \oplus R[1]$. Clearly, $P$ is proxy-small. Since $R$ is semi-Artinian, any $R$-module admits a presentation as a transfinite extension of copies of the simple module $k$. It follows that 
    \begin{displaymath}
        \overline{\langle P \rangle}^{(-\infty,0]} = D^{\leq 0}_{\operatorname{qc}}(R) = \overline{\langle R \rangle}^{(-\infty,0]}.
    \end{displaymath} 
    Hence, \eqref{thm:characterize_left_proxy_small1} of \Cref{thm:characterize_left_proxy_small} holds. However, we claim that $P$ is not $t$-proxy-small. To this aim, it is enough to show that $R \not\in \langle P \rangle^{(-\infty,0]}$. Let 
    \begin{displaymath}
        \mathcal{S} = \{X \in D^{\leq 0}_{\operatorname{qc}} (R) \mid \mathcal{H}^0(X) \text{ is of finite length}\}.
    \end{displaymath} 
    Then $\mathcal{S}$ is closed under suspension and direct summands. Also, a straightforward long exact sequence argument shows that $\mathcal{S}$ is closed under extensions. Clearly, $P \in \mathcal{S}$. So, with the previous observation, we have $\langle P \rangle^{(-\infty,0]} \subseteq \mathcal{S}$. Yet, $R$ is not of finite length, which implies $R \not\in \mathcal{S}$, proving the claim. Consequently, $P$ satisfies conditions \eqref{thm:characterize_left_proxy_small1} and \eqref{thm:characterize_left_proxy_small2} of \Cref{thm:characterize_left_proxy_small}, but fails \eqref{thm:characterize_left_proxy_small3}.
\end{example}

\section{\texorpdfstring{$t$}{}-proxy small \texorpdfstring{$\otimes$}{}-preaisles}
\label{sec:left_tensor_proxy_smallness}

This section extends the notion of $t$-proxy smallness in the presence of a tensor action. We introduce the desired extension of $t$-proxy smallness and identify a classification of associated aisles on $D_{\operatorname{qc}}(X)$ for a Noetherian scheme $X$. To do so, we compare them to subcategories of the singularity category $D_{\operatorname{sg}}(X)$ and Thomason filtrations on $X$. Here, we leverage \Cref{sec:classify_pseudocoherent_aisles} on aisles generated by coherent complexes on $D_{\operatorname{qc}}(X)$. Particularly, \Cref{thm:injective_mapping_for_t_proxy_tensor_aisles_scheme} establishes an embedding of such aisles, whose image is identified in \Cref{prop:image_assignment_tstr_general}.

\begin{definition}
        \label{def:t_proxy_tensor_small}
	In the presence of a tensor action $\odot\colon
	\mathcal{T}\times\mathcal{K}\to \mathcal{K}$ as in \Cref{sec:prelim_tensor}, we say that a subcategory
	$\mathcal{S}\subseteq\mathcal{K}$ is \textbf{$t$-$\odot$-proxy small} if the
	$\odot$-preaisle $\langle \mathcal{S}\rangle_\odot^{(-\infty,0]}$ is $t$-proxy
	small, or equivalently (by \Cref{lem:generated_odot_aisle}), if
	$\mathcal{P}^{\leq 0}\odot\mathcal{S}$ is $t$-proxy small. As per usual, an object $P \in \mathcal{K}$ is \textbf{$t$-$\odot$-proxy small} if the singleton subcategory $\{P\}$ is such.
\end{definition}

\begin{convention}
    We remind the reader of the notation in \Cref{ex:actions}. Let $\star$, $\odot$, and $\otimes$ denote the tensor actions of $\operatorname{Perf}(X)$ respectively on $D_{\operatorname{qc}}(X)$, $K(\operatorname{Inj}(X))$, and $S_{\operatorname{qc}}(X)$. In particular, we recall that a suspended subcategory of $D_{\operatorname{qc}}(X)$, (resp., $K(\operatorname{Inj}(X))$ or $S_{\operatorname{qc}}(X)$) is $\otimes$-suspended (resp., $\odot$-suspended or $\star$-suspended) if it is closed under the action of $\operatorname{Perf}^{\leq 0}(X)$. In what follows, \Cref{def:t_proxy_tensor_small} will be considered solely for objects belonging to $D^b_{\operatorname{coh}}(X)$, providing us with the notion of $t$-$\otimes$-proxy small objects. Finally, we leave as an exercise to the reader to check that a $\otimes$-preaisle $\mathcal{A}$ of $D^b_{\operatorname{coh}}(X)$ is $t$-proxy small if and only if there is a set $\mathcal{S} \subseteq D^b_{\operatorname{coh}}(X)$ of $t$-$\otimes$-proxy small objects such that $\mathcal{A} = \langle \mathcal{S} \rangle_{\otimes}^{(-\infty,0]}$.
\end{convention}

\begin{lemma}
    \label{lem:Krause_recollement_Q_compatible}
    Let $X$ be a Noetherian scheme. Then $Q$ is compatible with the actions $\odot$ and $\otimes$. That is, for any $A\in D_{\operatorname{qc}}(X)$ and $E\in K(\operatorname{Inj}(X))$, one has $Q(A\odot E)\cong A \otimes Q(E)$. 
\end{lemma}

\begin{proof}
    The first claim follows the same argument as in the affine case of \cite[Proposition 5.3]{Stevenson:2014b}. In fact, it holds for our setting because $D_{\operatorname{qc}}(X)$ acts on $K(\operatorname{Inj}(X))$ via $K$-flat resolutions (see e.g.\ Remark 3.7 of loc.\ cit.). 
\end{proof}

\begin{theorem}
    \label{thm:injective_mapping_for_t_proxy_tensor_aisles_scheme}
	Let $X$ be a Noetherian scheme. Then:
	\begin{enumerate}
		\item
			Let $\mathcal{S}\subseteq K^c(\operatorname{Inj}X)$ be a class of compact
			objects of $K(\operatorname{Inj}X)$. If
			$Q(\mathcal{S})\subseteq D^b_{\operatorname{coh}}(X)$ is $t$-proxy-small,
			then the $t$-structure generated by $\mathcal{S}$ restricts along Krause's
			recollement.
		\item The assignment
			$\mathcal{A}\mapsto (I_\lambda(\mathcal{A}),Q(\mathcal{A}))$
			yields an injective mapping:
			\begin{multline*}
				\{\,\text{$\odot$-aisles generated by }\mathcal{S}\subseteq K^c(\operatorname{Inj}X)\text{ with
				}Q(\mathcal{S})\subseteq D^b_{\operatorname{coh}}(X)\text{
					$t$-proxy-small}\,\} \to \\
				\{\,\text{compactly generated $\star$-aisles of }S_{\operatorname{qc}}(X)\,\}\times
				\{\,\text{compactly generated $\otimes$-aisles of }D_{\operatorname{qc}}(X)\,\}
			\end{multline*}
		\item Intersecting the aisles in (2) with the corresponding categories of
			compacts and using the equivalence $Q\colon
			K^c(\operatorname{Inj}(X))\simeq D^b_{\operatorname{coh}}(X):Q_\rho$, we obtain an injective mapping:
			\begin{multline*}
				\{\,\text{$t$-proxy small $\otimes$-suspended subcategories
				of } D^b_{\operatorname{coh}}(X)\,\} \to \\
				\{\,\text{$\star$-suspended subcategories of }D_{\operatorname{sg}}(X)\,\}\times
				\{\,\text{$\otimes$-suspended subcategories of }\operatorname{Perf}(X)\,\}
			\end{multline*}
			In particular, the second component of this last cartesian product is
			parametrised by Thomason filtrations on $X$.
	\end{enumerate}
\end{theorem}

\begin{proof}
	(1) We first recall the following \emph{d\'{e}vissage} argument, which we are
	going to use a couple of times. If $F\colon \mathcal{K}\to \mathcal{T}$ is a
	exact functor preserving coproducts and
	$\mathcal{A}\subseteq\mathcal{T}$ is a cocomplete preaisle  in $\mathcal{T}$,
	then the subcategory $F^{-1}(\mathcal{A})\colonequals\{\,k\in\mathcal{K}\mid
	F(k)\in\mathcal{A}\,\}$ is easily seen to be closed under positive shifts,
	extensions and coproducts. It follows that it will contain the smallest
	cocomplete preaisle generated by any subcategory
	$\mathcal{S}\subseteq\mathcal{K}$ for which
	$F(\mathcal{S})\subseteq\mathcal{A}$. In particular, for any subcategory
	$\mathcal{S}\subseteq\mathcal{K}$, by choosing $\mathcal{A}\colonequals\overline{\langle
	F(\mathcal{S}) \rangle}^{(-\infty,0]}$ we obtain that
	$F\left(\overline{\langle \mathcal{S} \rangle}^{(-\infty,0]}\right)\subseteq
	\overline{\langle F(\mathcal{S}) \rangle}^{(-\infty,0]}$.

	Now, assume as in the statement that $\mathcal{S}\subseteq
	K^c(\operatorname{Inj}X)$ is such that $Q(\mathcal{S})$ is $t$-proxy small.
	Let $\mathcal{P}\colonequals\overline{\langle Q(\mathcal{S}) \rangle}^{(-\infty,0]}\cap\operatorname{Perf}(X)$.
	By definition, we are assuming that
	$\overline{\langle \mathcal{P} \rangle}^{(-\infty,0]}=
	\overline{\langle Q(\mathcal{S}) \rangle}^{(-\infty,0]}$ and that
	$\mathcal{P}\subseteq \langle Q(\mathcal{S}) \rangle^{(-\infty,0]}$.
	By the properties of Krause's recollement, moreover we have that $Q_\rho Q$ is the
	identity on $K^c(\operatorname{Inj}(X))$ and that $Q_\lambda\equiv Q_\rho$ on
	$\operatorname{Perf}(X)$. We remark that both the exact functors $Q$
	and $Q_\lambda$, being left adjoints, preserves coproducts, allowing us to
	argue by d\'{e}vissage as explained above. Now, we have that:
	\begin{multline*}
		Q_\lambda Q\left(\overline{\langle \mathcal{S} \rangle}^{(-\infty,0]}\right) \subseteq
		Q_\lambda\left(\overline{\langle Q(\mathcal{S}) \rangle}^{(-\infty,0]}\right) =
		Q_\lambda\left(\overline{\langle \mathcal{P} \rangle}^{(-\infty,0]}\right) \subseteq
		\overline{\langle Q_\lambda(\mathcal{P}) \rangle}^{(-\infty,0]} = \\
		= \overline{\langle Q_\rho(\mathcal{P}) \rangle}^{(-\infty,0]}
		\subseteq \overline{\langle Q_\rho\left(\langle Q(\mathcal{S})
		\rangle^{(-\infty,0]}\right) \rangle}^{(-\infty,0]} \subseteq
		\overline{\langle\langle Q_\rho Q(\mathcal{S})
		\rangle^{(-\infty,0]} \rangle}^{(-\infty,0]} =
		\overline{\langle \mathcal{S} \rangle}^{(-\infty,0]}.
	\end{multline*}
	By \Cref{lem:glueing_recollement}, we conclude that the $t$-structure of
	$K(\operatorname{Inj}(X))$ generated by $\mathcal{S}$ restricts along Krause's
	recollement.

	(2) By \Cref{lem:glueing_recollement}, the assignment in the statement is injective from
	the class of aisles of $K(\operatorname{Inj}(X))$ to the cartesian product of
	the classes of aisles
	of $S_{\operatorname{qc}}(X)$ and $D_{\operatorname{qc}}(X)$.
	We need to show that if we start with a compactly generated $\odot$-aisle
	$\mathcal{A}$ of $K(\operatorname{Inj}(X))$, both $I_\lambda(\mathcal{A})$ and
	$Q(\mathcal{A})$ are compactly generated tensor-aisles, with respect to
	$\star$ and $\otimes$ respectively.

	Let $\mathcal{S}\colonequals\mathcal{A}\cap K^c(\operatorname{Inj}(X))$ be the
	class of compact generators of $\mathcal{A}$. We have:
	\begin{displaymath}
        I_\lambda(\mathcal{A})=\overline{\langle I_\lambda(\mathcal{S})
			\rangle}^{(-\infty,0]} \text{ in }D_{\operatorname{sg}}(X)
		\quad\text{ and }\quad
		Q(\mathcal{A})=\overline{\langle Q(\mathcal{S})
			\rangle}^{(-\infty,0]} \text{ in }D_{\operatorname{qc}}(X).
    \end{displaymath}
    respectively. Indeed, in these two equalities, the inclusions
	$(\subseteq)$ follows by d\'{e}vissage, since $I_\lambda$ and $Q$ are triangulated
	and coproduct preserving. The other inclusions $(\supseteq)$ follow from the
	fact that since $\mathcal{A}$ restricts, the left-hand sides are cocomplete
	preaisle.

	From this, we see that $I_\lambda(\mathcal{A})$ is automatically compactly
	generated, since $I_\lambda$, having a coproduct preserving right adjoint $I$,
	preserves compacts. To show that the other aisle is compactly generated, we
	use the tensor compatibility.
	Since $\mathcal{A}$ is a $\odot$-aisle, then $\mathcal{S}$ is a
	$\odot$-preaisle of $K^c(\operatorname{Inj}(X))$, see
	\Cref{prop:suspended_aisles_to_compactly_generated_aisles_a_la_neeman}. By
	definition of the $\odot$-action (see \Cref{ex:actions}), this means that
	$Q(\mathcal{S})$ is a $\otimes$-preaisle of $D^b_{\operatorname{coh}}(X)$, and
	therefore by \Cref{lem:generated_odot_aisle} we have that:
	\begin{displaymath}
        \overline{\langle Q(\mathcal{S}) \rangle}^{(-\infty,0]} = 
		\overline{\langle Q(\mathcal{S}) \rangle}^{(-\infty,0]}_\otimes.
    \end{displaymath}
	is a compactly generated aisle of $D_{\operatorname{qc}}(X)$. This also shows
	that this aisle is a $\otimes$-aisle, and one shows similarly that $I_\lambda(\mathcal{A})$
	is a $\star$-aisle.

	(3) This follows from (2) using
	\Cref{prop:suspended_aisles_to_compactly_generated_aisles_a_la_neeman}.
\end{proof}

\begin{proposition}
    \label{prop:image_assignment_tstr_general}
    Let $(\mathcal{B},\phi)$ be an element of the target for $\Theta$ in \Cref{introthm:injective_mapping_for_t_proxy_tensor_aisles_scheme}. Then $(\mathcal{B},\phi)$ is in the image of $\Theta$ if, and only if, for any $B \in \mathcal{B}$ there is $M \in D^b_{\operatorname{coh}}(X)$ such that
    \begin{enumerate}
        \item \label{prop:image_assignment_tstr_general1} $\pi(M) \in \mathcal{B}$,
        \item \label{prop:image_assignment_tstr_general2} $B \in \langle \pi (M) \rangle^{(-\infty,0]}_\otimes$,
        \item \label{prop:image_assignment_tstr_general3} $\operatorname{supp}(\mathcal{H}^i(M)) \subseteq \phi(i)$ for all $i \in \mathbb{Z}$.
    \end{enumerate}
\end{proposition}

\begin{proof}
    To set the stage,
    \begin{itemize}
        \item $\overline{\langle \mathcal{A} \rangle}^{(-\infty,0]}_\star$ is the $\star$-aisle of $S_{\operatorname{qc}}(X)$ generated by $\mathcal{A}$;
        \item $\mathcal{U}$ is the $\otimes$-aisle of $D_{\operatorname{qc}}(X)$ corresponding to $\phi$;
        \item $\mathcal{Z}$ is the aisle of $K(\operatorname{Inj}(X))$ obtained by glueing $\overline{\langle \mathcal{A} \rangle}^{(-\infty,0]}_\star$ and $\mathcal{U}$, i.e.\ 
        \begin{displaymath}
            \mathcal{Z} \colonequals \left\{ E \in K(\operatorname{Inj}(X)) \mid I_\lambda(E) \in \overline{\langle \mathcal{A} \rangle}^{(-\infty,0]}_\star, Q(E) \in \mathcal{U} \right\}.
        \end{displaymath}
    \end{itemize}

    Before proving the desired claim, we show that 
    \begin{equation}
        \tag{$\circ$}
        \label{eq:image_assignment_tstr_general_claim}
        \textrm{$(\mathcal{A},\phi)$ is in the image of $\Theta$ if, and only if, $\mathcal{Z}$ is compactly generated.}
    \end{equation}
    Observe, if $(\mathcal{A},\phi)$ is in the image of $\Theta$, then there is $\odot$-suspended $\mathcal{S}\subseteq D^b_{\operatorname{coh}}(X)$ that is $t$-proxy small in $D_{\operatorname{qc}}(X)$ whose image under $\Theta$ coincides with $(\mathcal{A},\phi)$. However, as $\Theta$ factors through 
    \begin{displaymath}
        \begin{aligned}
            &\left\{ \textrm{compactly generated } \star\textrm{-aisles on } S_{\operatorname{qc}}(X) \right\} 
            \\&\times \left\{ \textrm{compactly generated } \otimes\textrm{-aisles on } D_{\operatorname{qc}}(X) \right\},
        \end{aligned}
    \end{displaymath}
    we see that the image of $\mathcal{S}$ under the assignment
    \begin{displaymath}
        \overline{\langle \mathcal{S} \rangle}^{(-\infty,0]}_{\odot} \mapsto \bigg( \overline{\langle I_\lambda Q_\rho (\mathcal{S}) \rangle}^{(-\infty,0]}_\star, \overline{\langle \mathcal{S} \rangle}^{(-\infty,0]}_\otimes \bigg).
    \end{displaymath}
    corresponds to $\mathcal{Z}$. So, $\mathcal{Z}$ must be compactly generated in such a case.

    Conversely, if $\mathcal{Z}$ is compactly generated, then \Cref{prop:suspended_aisles_to_compactly_generated_aisles_a_la_neeman} tells us we can find a $\odot$-suspended subcategory $\mathcal{P}$ of $K(\operatorname{Inj}(X))^c$ such that $\mathcal{Z} = \overline{\langle \mathcal{P} \rangle}^{(-\infty,0]}_{\odot}$. As $Q$ is compatible with the tensor action and restricts to an equivalence $K(\operatorname{Inj}(X))^c \to D^b_{\operatorname{coh}}(X)$, we see that $Q(\mathcal{P})$ is an $\otimes$-suspended category on $D^b_{\operatorname{coh}}(X)$. It suffices to show that the $\otimes$-suspended subcategory $Q(\mathcal{P})$ is $t$-proxy small in $D_{\operatorname{qc}}(X)$. Choose $A\in \operatorname{Perf}(X) \cap \overline{\langle Q(\mathcal{P}) \rangle}^{(-\infty,0]}_{\otimes}$. 
    Using \Cref{cor:compactly_generated_odot_aisle}, we know that 
    \begin{displaymath}
        \overline{\langle \mathcal{P} \rangle}^{(-\infty,0]}_{\odot} = \overline{\langle (\operatorname{Perf}(X) \cap D^{\leq 0}_{\operatorname{qc}}(X) )\odot \mathcal{P} \rangle}^{(-\infty,0]}.
    \end{displaymath}
    As $Q_\lambda$ is right $t$-exact, we know that $Q_\lambda (\mathcal{U})\subseteq \mathcal{X}$; whereas $Q$ is $t$-exact, so $Q(\mathcal{Z})\subseteq \mathcal{U}$. However, as $Q_\lambda$ is fully faithful, we have $Q(\mathcal{Z})= \mathcal{U}$.
    Hence, $\mathcal{U}$ is compactly generated by $Q(\mathcal{Z})\cap \operatorname{Perf}(X)$, and so
    \begin{displaymath}
        \begin{aligned}
            A \in \operatorname{Perf}(X) \cap \overline{\langle Q(\mathcal{P}) \rangle}^{(-\infty,0]}_{\otimes} \subseteq \operatorname{Perf}(X) \cap \overline{\langle Q((\operatorname{Perf}(X) \cap D^{\leq 0}_{\operatorname{qc}}(X) )\odot \mathcal{P}) \rangle}^{(-\infty,0]}
        \end{aligned}
    \end{displaymath}
    Here, we have used the fact that $Q$ is compatible with the actions and preserves coproducts. Clearly, 
    \begin{displaymath}
        Q((\operatorname{Perf}(X) \cap D^{\leq 0}_{\operatorname{qc}}(X) )\odot \mathcal{P}) \subseteq D^b_{\operatorname{coh}}(X).
    \end{displaymath}
    So, by \cite[Theorem A.7]{Keller/Nicolas:2013}, there is a distinguished triangle
    \begin{displaymath}
        \bigoplus_{i\geq 1} P_i \to A \to (\bigoplus_{i\geq 1} P_i) [1] \to (\bigoplus_{i\geq 1} P_i)[1]
    \end{displaymath}
    where each $P_i$ is an $i$-fold extension of small coproducts of nonnegative shifts of objects in
    \begin{displaymath}
        \operatorname{Perf}(X) \cap Q((\operatorname{Perf}(X) \cap D^{\leq 0}_{\operatorname{qc}}(X) )\odot \mathcal{P}).
    \end{displaymath}
    Indeed, this subcategory compactly generates $\mathcal{U}$. Now, applying $Q_\lambda$ to this distinguished triangle, we can check that each object in it belongs to $\mathcal{Z}$. 
    Hence, $Q_\lambda (A)\in \mathcal{Z}$. Yet, $Q_\lambda (A)\in K(\operatorname{Inj}(X))^c$ as $A$ is compact object and $Q_\lambda$ preserves compacts. So, $Q_\lambda (A) \in \mathcal{P}$. Hence, applying $Q$ once more finishes the claim. All in all, we have now proved \eqref{eq:image_assignment_tstr_general_claim}. 

    Now, we return the proof for our desired claim. First, we consider the case where $(\mathcal{A},\phi)$ is in the image of $\Theta$. Choose $A \in \mathcal{A}$. Then $I(A) \in \mathcal{Z}$, and so, $I(A) \in \overline{\langle \mathcal{Z} \cap K(\operatorname{Inj}(X))^c \rangle}^{(-\infty,0]}$. 
    Hence, $A \in \overline{\langle I_\lambda (\mathcal{Z} \cap K(\operatorname{Inj}(X))^c) \rangle}^{(-\infty,0]}$, 
    which implies $A\in \langle I_\lambda(Z) \rangle^{(-\infty,0]}_{\star}$ for some $Z \in \mathcal{Z} \cap K(\operatorname{Inj}(X))^c$. 
    By \cite[Remark 3.8]{Krause:2005}, we can find $M \in D^b_{\operatorname{coh}}(X)$ such that $Z = Q_\rho(M)$. Then \eqref{prop:image_assignment_tstr_general1} and \eqref{prop:image_assignment_tstr_general3} as $Z \in \mathcal{Z} \cap K(\operatorname{Inj}(X))^c$, 
    whereas \eqref{prop:image_assignment_tstr_general2} holds by the construction.

    Next, we check the converse direction. From glueing, it follows that
    \begin{displaymath}
        \mathcal{Z} = \overline{\langle I(\mathcal{A}) \cup Q_\lambda(\mathcal{U} \cap \operatorname{Perf}(X)) \rangle}^{(-\infty,0]}.
    \end{displaymath}
    Choose $A \in \mathcal{A}$. Define $M = M_A$ to be as in the condition. Consider the distinguished triangle
    \begin{displaymath}
        Q_\lambda(M) \to Q_\rho(M) \to I I_\lambda Q_\rho(M) \to Q_\lambda(M) [1].
    \end{displaymath}
    By \eqref{prop:image_assignment_tstr_general1} and \eqref{prop:image_assignment_tstr_general3}, the left and the right vertices of the triangle belong to $\mathcal{Z}$, and thus so does $Q_\rho(M)$. Consequently, $\mathcal{Z}$ is compactly generated by $\{Q_\rho(M_A) \mid A \in \mathcal{A}\} \cup Q_\lambda(\mathcal{U}^c)$.
\end{proof}

\section{Characterizations}
\label{sec:characterization}

This section contains applications of the theory built up in earlier sections. 

\subsection{\texorpdfstring{$\otimes$}{ }-proxy smallness in $D^b_{\operatorname{coh}}(X)$}
\label{sec:characterization_tensor_proxy_small}

Some of below is spelled out in detail. It might be known to `experts' but we find at times the literature lacking details. So, we make it clear for the reader. See \cite{Stevenson:2013} for some details.



\begin{lemma}
    \label{lem:support_to_localizing_tensor_subcategories}
    Let $X$ be a Noetherian scheme. For any $E,G \in D_{\operatorname{qc}}(X)$, we have $\overline{\langle E \rangle}_\otimes \subseteq \overline{\langle G \rangle}_\otimes$ if, and only if, $\operatorname{supp}(E)\subseteq \operatorname{supp}(G)$. 
\end{lemma}

\begin{proof}
    This is essentially \cite[Corollary 4.13]{AlonsoTarrio/JeremiasLopez/SoutoSalorio:2004} with \Cref{lem:generated_odot_aisle}. Note that `rigid' in loc.\ cit.\ means $\otimes$-closed in our case above. See \cite[Proposition A.2]{Iyengar/Lipman/Neeman:2015} for a reformulation. 
\end{proof}

\begin{lemma}
    [cf.\ {\cite[\S 6.1]{Briggs/Iyengar/Letz/Pollitz:2022}}]
    \label{def:virtually_small}
    Let $X$ be a Noetherian scheme and $P$ be a compact generator for $D_{\operatorname{qc}}(X)$. An object $E$ in $D^b_{\operatorname{coh}}(X)$ is \textbf{$\otimes$-proxy small} if and only if $\langle P\otimes^{\mathbf{L}} E \rangle$ contains a perfect complex with same support as $E$.
\end{lemma}

\begin{proof}
    If $E$ is $\otimes$-proxy small then $\overline{\langle E \rangle}_\otimes$ is compactly generated by $\mathcal{S} = \overline{\langle E \rangle}_\otimes \cap \operatorname{Perf}(X)$, meaning that $\overline{\langle E \rangle}_\otimes = \overline{\langle \mathcal{S} \rangle}_\otimes$. It follows that $\operatorname{supp}(E) = \bigcup_{S \in \mathcal{S}} \operatorname{supp}(S)$. Since $\operatorname{supp}(E)$ is closed in $X$ and $X$ is Noetherian, we have $\operatorname{supp}(E) = \bigcup_{S \in \mathcal{S'}} \operatorname{supp}(S)$ for a finite subset $\mathcal{S}'$ of $\mathcal{S}$. Setting $C = \coprod_{S \in \mathcal{S}'}S$ we obtain a perfect complex with $\operatorname{supp}(C) = \operatorname{supp}(E)$. By the definition of $\otimes$-proxy smallness we also have $C \in \langle P\otimes^{\mathbf{L}} E \rangle$.

    For the converse, assume that $C \in \langle P\otimes^{\mathbf{L}} E \rangle$ is a perfect complex with $\operatorname{supp}(C) = \operatorname{supp}(E)$. Then \Cref{lem:support_to_localizing_tensor_subcategories} yields $\overline{\langle C \rangle}_\otimes = \overline{\langle E \rangle}_\otimes$, which in turn implies that $\overline{\langle E \rangle}_\otimes = \overline{\langle \langle E \rangle_\otimes \cap \operatorname{Perf}(X) \rangle}_\otimes$, which verifies that $E$ is $\otimes$-proxy small.
\end{proof}

\begin{lemma}
    \label{lem:psuedo_coherent_approx_triangle_fixed_support}
    Let $X$ be a Noetherian scheme. For any $E\in D^-_{\operatorname{coh}}(X)$, there is a distinguished triangle
    \begin{displaymath}
        \bigoplus_{n\geq 1} P_n \to \bigoplus_{n\geq 1} P_n \to E \to (\bigoplus_{n\geq 1} P_n)[1]
    \end{displaymath}
    where each $P_n \in \operatorname{Perf}(X)$ and $\operatorname{supp}(P_n)\subseteq \operatorname{supp}(E)$. Additionally, if $E\in D^b_{\operatorname{coh}}(X)$, then it is possible to find an $N\geq 1$ such that $\operatorname{supp}(P_N) = \operatorname{supp}(E)$.
\end{lemma}

\begin{proof}
    This essentially follows from `approximation by perfect complexes' \`{a} la Lipman--Neeman \cite[Theorem 4.1]{Lipman/Neeman:2007}. We use the argument for the $(1)\implies (2)$ direction in \cite[\href{https://stacks.math.columbia.edu/tag/0DJM}{Tag 0DJM}]{StacksProject}. 
    In fact, the same strategy works where we can modify each step of loc.\ cit.\ to use an approximation with the support of $K_n$ contained in $\operatorname{supp}(E)$ (which is possible by \cite[\href{https://stacks.math.columbia.edu/tag/08EL}{Tag 08EL}]{StacksProject}).
    The last claim follows by shifting $E$, if needed, suitably so that $P_1 \to E$ induces an isomorphism cohomology in degrees where $E$ is nonzero (which is possible because $E$ is bounded).
\end{proof}

\begin{lemma}
    \label{lem:pullback_proxy_smallness_along_flat}
    Let $f\colon Y \to X$ be a flat morphism of Noetherian schemes. Suppose $E\in D_{\operatorname{qc}}(X)$. If $E$ is $\otimes$-proxy small, then $\mathbf{L} f^\ast E \in D_{\operatorname{qc}}(Y)$ is $\otimes$-proxy small.
\end{lemma}

\begin{proof}
    From $f$ being flat, we know that $\mathbf{L} f^\ast D^b_{\operatorname{coh}}(X)\subseteq D^b_{\operatorname{coh}}(Y)$. Indeed, $\mathbf{L} f^\ast$ preserves pseudocoherent complexes, whereas flatness implies $f^\ast$ is exact on $\operatorname{Qcoh}$. Additionally, $\mathbf{L} f^\ast$ commutes with $\otimes^{\mathbf{L}}$. Let $P^\prime$ be a compact generator for $D_{\operatorname{qc}}(Y)$ and $P$ be a compact generator for $D_{\operatorname{qc}}(X)$. Suppose $P\otimes^{\mathbf{L}} E$ finitely builds a perfect complex $Q$ with same support as $E$ where $E\in D^b_{\operatorname{coh}}(X)$. Clearly, $\mathbf{L} f^\ast P\in \langle P^\prime \rangle$ and $\mathbf{L} f^\ast Q \in \langle \mathbf{L} f^\ast P \otimes^{\mathbf{L}} \mathbf{L} f^\ast E\rangle$, so $\mathbf{L} f^\ast Q \in \langle P^\prime \otimes^{\mathbf{L}} \mathbf{L} f^\ast E \rangle$. We need to check that $\operatorname{supp}(\mathbf{L}f^\ast(E)) = f^{-1}(\operatorname{supp}(Q))$. Observe $\mathbf{L} f^\ast Q \in \langle P^\prime \otimes^{\mathbf{L}} \mathbf{L} f^\ast E \rangle$ implies $f^{-1}(\operatorname{supp}(Q)) \subseteq \operatorname{supp}(\mathbf{L}f^\ast E)$. On the other hand, \cite[Lemma 2.1]{Lank:2026} gives us $\operatorname{supp}(\mathbf{L}f^\ast (E)) \subseteq f^{-1}(\operatorname{supp}(Q))$.
\end{proof}

\begin{reminder}
    A Noetherian local ring $(R,\mathfrak{m})$ is called a \textbf{complete intersection} if its $\mathfrak{m}$-adic completion $\widehat{R}$ is of the form $Q/(f_1,\ldots,f_c)$ where $Q$ is a regular local ring and $f_1,\ldots,f_c\in Q$ is a regular sequence. More generally, a Noetherian scheme is called a \textbf{local complete intersection} if $\mathcal{O}_{X,p}$ is a complete intersection for all $p\in X$. Often, we might say `let $X$ be a Noetherian scheme which is locally a complete intersection'.
\end{reminder}

\begin{proposition}
    \label{prop:LCI_iff_proxy_small}
    Let $X$ be a Noetherian scheme. Then every object of $D^b_{\operatorname{coh}}(X)$ is $\otimes$-proxy small if, and only if, $X$ is locally a complete intersection.
\end{proposition}

\begin{proof}
    First we assume that every object of $D^b_{\operatorname{coh}}(X)$ is $\otimes$-proxy small. Choose $p\in X$. Let $s\colon \operatorname{Spec}(\mathcal{O}_{X,p}) \to X$ be the canonical morphism. Note that $\mathbf{L} s^\ast \colon D^b_{\operatorname{coh}}(X) \to D^b_{\operatorname{coh}}(\mathcal{O}_{X,p})$ is is essentially surjective (e.g.\ use \cite[Lemma 3.9]{Letz:2021} and \cite[Theorem 4.4]{Elagin/Lunts/Schnurer:2020}). Moreover, \Cref{lem:pullback_proxy_smallness_along_flat,def:virtually_small} imply that $\mathbf{L} s^\ast E$ is proxy small for every $E\in D^b_{\operatorname{coh}}(X)$ (of course, using our hypothesis). Tying things together, we know that every object of $D^b_{\operatorname{coh}}(\mathcal{O}_{X,p})$ is proxy small, so \cite[Theorem 5.2]{Pollitz:2019} ensures $\mathcal{O}_{X,p}$ is a local complete intersection. However, $p\in X$ was arbitrary, so $X$ is locally a complete intersection.

    Now we show the converse direction where $X$ is locally a complete intersection. Let $E\in D^b_{\operatorname{coh}}(X)$. We claim that $E$ is $\otimes$-proxy small. There is nothing to check when $E$ is the zero object, so we can impose this is not so. Choose $Q\in \operatorname{Perf}(X)$ satisfying $\operatorname{supp}(E)=\operatorname{supp}(Q)$. For each $p\in X$, we have $Q_p \in \langle E_p \rangle$. Indeed, we only need to concern ourselves with $p\in \operatorname{supp}(E)$. From $\mathcal{O}_{X,p}$ being a local complete intersection, we know that there is $B\in \operatorname{Perf}(\mathcal{O}_{X,p})$ such that $\operatorname{supp}(B)=\operatorname{supp}(E_p)$ and $B\in \langle E_p \rangle$. However, $\operatorname{supp}(Q_p) = \operatorname{supp}(E_p)$ (see e.g.\ \cite[Lemma 4.8]{Hall/Rydh:2017}), so \cite[Lemma 1.2]{Neeman:1992} ensures $Q_p \otimes^{\mathbf{L}} K(p)\in \langle B \rangle \subseteq \langle E_p \rangle$ where $K(p)$ is the Koszul complex on a minimal set of generators for the maximal ideal of $\mathcal{O}_{X,p}$. By \cite[Theorem 1.7]{BILMP:2023}, it follows that $Q\in \langle P \otimes^{\mathbf{L}} E \rangle$, so \Cref{lem:generated_odot_aisle} implies $E$ is $\otimes$-proxy small. From $E$ being an arbitrary object of $D^b_{\operatorname{coh}}(X)$, the desired claim follows.
\end{proof}

\subsection{Result}
\label{sec:characterization_result}

We give a characterization of Noetherian schemes which are locally complete intersections in terms of our notion of $t$-$\otimes$-proxy smallness.

\begin{remark}
    \label{rmk:internal-hom}
    In the proof of \Cref{lem:stalk-t-exact} that follows, we shall make use of some useful properties of the sheaf Hom complex that we recall now, the main reference being \cite[\S 3]{Murfet:2006}. The sheaf Hom complex $\operatorname{\mathcal{H}\! \mathit{om}}(-,-)$ is built using the internal sheaf $\operatorname{\mathcal{H}\! \mathit{om}}$ of $\mathcal{O}_X$-modules and we shall consider it naturally as a functor 
    $$\operatorname{\mathcal{H}\! \mathit{om}}(-,-): K(X)^{op} \times K(X) \to K(X)$$ 
    on the homotopy category $K(X)$ of $\mathcal{O}_X$-modules. Recall that $\operatorname{\mathcal{H}\! \mathit{om}}(S,E)$ belongs to $K(\operatorname{Qcoh}(X))$ whenever $S \in K(\operatorname{coh}(X))$ and $E \in K(\operatorname{Qcoh}(X))$, see e.g. \cite[01CQ, 0GMV]{StacksProject}. By \cite[Proposition 28 \& Definition 5]{Murfet:2006} or \cite[Lemma 2.4.5.1]{Lipman:2009}, we obtain the following isomorphism for any $i \in \mathbb{Z}$:
    \begin{displaymath}
        \operatorname{Hom}_{K(\operatorname{Qcoh}(X))}(S,E[i]) \cong \operatorname{Hom}_{\operatorname{Qcoh}(X)}(\mathcal{O}_X,\mathcal{H}^i\operatorname{\mathcal{H}\! \mathit{om}}(S,E)).
    \end{displaymath}
\end{remark}

\begin{lemma}
    \label{lem:stalk-t-exact}
    Let $\mathcal{T}^{\leq 0}$ be a compactly generated $\odot$-aisle on $K(\operatorname{Inj}(X))$ with associated coaisle $\mathcal{T}^{\geq 0}$. Then for any $E \in K(\operatorname{Inj}(X))$:
    \begin{enumerate}
        \item \label{lem:stalk-t-exact1}$E \in \mathcal{T}^{\geq 0}$ if, and only if, $E_x \in \mathcal{T}^{\geq 0}$ for all $x \in X$,
        \item \label{lem:stalk-t-exact2} $E \in \mathcal{T}^{\leq 0}$ if, and only if, $E_x \in \mathcal{T}^{\leq 0}$ for all $x \in X$.
    \end{enumerate}
    In particular, $(-)_x \colon K(\operatorname{Inj} X) \to K(\operatorname{Inj} X)$ is $t$-exact with respect to the $t$-structure $(\mathcal{T}^{\leq 0},\mathcal{T}^{\geq 0})$.
\end{lemma}

\begin{proof}
    Choose $\mathcal{S} \subseteq D_{\operatorname{coh}}^b(X)$ satisfying $\mathcal{T}^{\leq 0} = \overline{\langle \mathcal{P} \rangle}^{(-\infty,0]}_{\odot}$ where
    \begin{displaymath}
        \mathcal{P} = \{Q_\rho(S) \mid S \in \mathcal{S}\} \subseteq K(\operatorname{Inj}(X))^c.
    \end{displaymath}
    By \Cref{cor:compactly_generated_odot_aisle}, we may impose that $\mathcal{S}$ is closed under the $\odot$-action of $\operatorname{Perf}^{\leq 0}(X)$, i.e.\ $\mathcal{T}^{\leq 0} = \overline{\langle \mathcal{P} \rangle}^{(-\infty,0]}$. Now, for any $S \in \mathcal{S}$ and $E \in K(\operatorname{Inj}(X))$, we have 
    \begin{displaymath}
        \operatorname{Hom}_{K(\operatorname{Qcoh}(X))}(S,E) \cong \operatorname{Hom}_{K(\operatorname{Inj}(X))}(Q_\rho(S),E).
    \end{displaymath}
    This ensures that
    \begin{displaymath}
        \mathcal{T}^{\geq 0} = \{E \in K(\operatorname{Inj}(X)) \mid \forall S \in \mathcal{S}, \operatorname{Hom}_{K(\operatorname{Qcoh}(X))}(S,E) = 0\}.
    \end{displaymath}
    To see the isomorphism on hom-sets, use the observation that the cone $C$ of $S \to Q_\rho(S)$ is a bounded below acyclic complex. In particular, it implies $\operatorname{Hom}_{K(\operatorname{Qcoh}(X))}(C,E) = 0$ for any $E \in K(\operatorname{Inj}(X))$, see e.g. \cite[Proposition 6.9 \& Example 6.15]{Stovicek:2014}.

    Note that $\mathcal{T}^{\leq 0}$ is closed under the $\odot$-action of $D^{\leq 0}$. So, it follows that $E \in \mathcal{T}^{\leq 0}$ implies $E_x \in \mathcal{T}^{\leq 0}$ for any $x \in X$. We claim that $E \in \mathcal{T}^{> 0}$ if, and only if, $\mathcal{H}^i\operatorname{\mathcal{H}\! \mathit{om}}(S,E) = 0$ for all $i \leq 0$ and for all $S \in \mathcal{S}$ (see \Cref{rmk:internal-hom}). As per \Cref{rmk:internal-hom}, we have the isomorphism $\operatorname{Hom}_{K(\operatorname{Qcoh}(X))}(S,E[i]) \cong \operatorname{Hom}_{\operatorname{Qcoh}(X)}(\mathcal{O}_X,\mathcal{H}^i\operatorname{\mathcal{H}\! \mathit{om}}(S,E)),$
    and so $\mathcal{H}^i\operatorname{\mathcal{H}\! \mathit{om}}(S,E) = 0$ for all $i \leq 0$ implies $E \in \mathcal{S}^{\perp_{\leq 0}}=\mathcal{T}^{> 0}$. For the converse, let $E$ be such that $\mathcal{H}^i\operatorname{\mathcal{H}\! \mathit{om}}(S,E) \neq 0$ for some $S \in \mathcal{S}$ and $i \leq 0$. By \cite[\href{https://stacks.math.columbia.edu/tag/08EL}{Tag 08EL}]{StacksProject}, there is $C \in \operatorname{Perf}^{\leq 0}(X)$ such that 
    \begin{displaymath}
        \operatorname{Hom}_{K(\operatorname{Qcoh}(X))}(C,\operatorname{\mathcal{H}\! \mathit{om}}(S,E)) \not\cong 0.
    \end{displaymath} 
    Indeed, there is a coherent sheaf $\mathcal{F}$ and a morphism $\mathcal{F} \to  Z^i(\operatorname{\mathcal{H}\! \mathit{om}}(S,E))$ to the $i$-th cocycle sheaf of $\operatorname{\mathcal{H}\! \mathit{om}}(S,E)$ such that the induced morphism $\mathcal{F} \to \mathcal{H}^i(\operatorname{\mathcal{H}\! \mathit{om}}(S,E))$ is nonzero. Choose $C \in \operatorname{Perf}^{\leq i}$ such that $\mathcal{H}^i(C) \cong \mathcal{F}$. Then the composition of the truncation of $C \to \mathcal{F}[-i]$ with the morphism $\mathcal{F}[-i] \to Z^i(\operatorname{\mathcal{H}\! \mathit{om}}(S,E))[-i] \to \operatorname{\mathcal{H}\! \mathit{om}}(S,E)$ is nonzero on the $i$-th cohomology. By adjunction, we obtain nonvanishing
    \begin{displaymath}
        \operatorname{Hom}_{K(\operatorname{Qcoh}(X))}(C \otimes_X S,E) \neq 0.
    \end{displaymath}
    which implies $E \not\in \mathcal{T}^{>0}$ as $C \otimes_X S \in \mathcal{S}$ by the assumption. This establishes the claim. Using \cite[\S III, Proposition 6.8]{Hartshorne:2013}, we know that $\mathcal{H}^i\operatorname{\mathcal{H}\! \mathit{om}}(S,E) = 0$ if, and only if, for each $x \in X$, one has
    \begin{displaymath}
        0 = \mathcal{H}^i(\operatorname{\mathcal{H}\! \mathit{om}}(S,E))_x \cong \mathcal{H}^i(\operatorname{\mathcal{H}\! \mathit{om}}(S,E)_x) \cong \mathcal{H}^i(\operatorname{\mathcal{H}\! \mathit{om}}(S,E_x)).
    \end{displaymath}
    Thus, \eqref{lem:stalk-t-exact1} follows.

    Finally, we need to show that $E_x \in \mathcal{T}^{\leq 0}$ for all $x \in X$ implies $E \in \mathcal{T}^{\leq 0}$. Consider the truncation triangle with respect to $(\mathcal{T}^{\leq 0},\mathcal{T}^{\geq 0})$:
    \begin{displaymath}
        E^{\leq 0} \to E \to E^{\geq 0} \to (E^{\leq 0})[1].
    \end{displaymath}
    By the previous paragraph, we already know that $(-)_x$ is t-exact. This implies that
    \begin{displaymath}
        (E^{\leq 0})_x \to E_x \to (E^{\geq 0})_x \to (E^{\leq 0}_x)[1]
    \end{displaymath}
    is the truncation triangle of $E_x$. Note that $(E^{\geq 0})_x =0$ for all $x \in X$ by the assumption. We claim that this implies $Y = E^{\geq 0} = 0$. Indeed, for any $S \in D_{\operatorname{coh}}^b(X)$ and $x \in X$ we have $\operatorname{\mathcal{H}\! \mathit{om}}(S,Y)_x \cong \operatorname{\mathcal{H}\! \mathit{om}}(S,Y_x) = 0$, implying that $\mathcal{H}^0\operatorname{\mathcal{H}\! \mathit{om}}(S,Y) = 0$. Then $\operatorname{Hom}_{K(\operatorname{Qcoh}(X))}(S,Y)=0$ for any $S \in D_{\operatorname{coh}}^b(X)$, which by the above implies vanishing of $\operatorname{Hom}_{K(\operatorname{Inj}(X))}(P,Y)$ for any $P \in K(\operatorname{Inj}(X))^c$, resulting in $Y=0$.
\end{proof}

\begin{proposition}
    \label{prop:local-CI-tproxy}
	Let $R$ be a local complete intersection ring. Then every object of $D^b_{\operatorname{coh}}(R)$ is $t$-proxy small in $D_{\operatorname{qc}}(R)$.
\end{proposition}

\begin{proof}
	Choose $E\in D^b_{\operatorname{coh}}(R)$. By the
	proof of \cite[Theorem 3.2]{Bergh:2009}, there exist positive integers $n_1,\dots,n_t$ and $E_0, E_1,\dots, E_t\in D^b_{\operatorname{coh}}(R)$ with $E_0\cong E$ and $E_t$ is compact. These objects fit in distinguished triangles,
	\begin{displaymath}
        E_{i-1}\to E_{i-1} [n_i] \to E_{i}\to E_{i-1}[1] 
    \end{displaymath}
	for $1\leq i\leq t$. For $n_i>0$, the distinguished triangle
	\begin{displaymath}
        E_{i-1} [n_i-1] \to E_i [-1]\to E_{i-1}\to E_{i-1} [n_i]
    \end{displaymath}
	tells us that $E_i[-1]\in \langle E_{i-1} \rangle^{[-\infty,0]}$. Hence, using
	induction, we can conclude that $P\colonequals E_t [-t]\in 
	\langle E \rangle^{[-\infty,0]} \cap \operatorname{Perf}(R)$ (i.e.\ is compact). 
    
    Now, we want to show that $E\in \overline{\langle P \rangle}^{(-\infty,0]}$. Towards that end, we leverage the classification of compactly generated aisles in $D_{\operatorname{qc}}(R)$ with Thomason filtrations on $\operatorname{Spec}(R)$; see \Cref{sec:prelim_Thomason} and references therein. Denote the Thomason filtration corresponding to $\overline{\langle P \rangle}^{(-\infty,0]}$ by $\phi_P$. Recall that
	$B\in \overline{\langle P \rangle}^{(-\infty,0]}$ if, and only if, $\phi_B(i)\subseteq \phi_P(i)$ for every
	$i\in\mathbb{Z}$. We claim that $\phi_{E_{j-1}}(i)\subseteq \phi_{E_j [-1]}(i)$ for $1\leq j\leq t$. By induction, it would follow that $\phi_E(i)\subseteq \phi_P(i)$ for every $i$, and hence, $E\in \overline{\langle P \rangle}^{(-\infty,0]}$.

	So, we prove the desired claim. Let $\mathfrak{p}\not\in \phi_{E_j [-1]}(i)$. This means that $0=(\mathcal{H}^n(\Sigma^{-1}E_j))_\mathfrak{p}=(\mathcal{H}^{n-1}(E_j))_\mathfrak{p}$ for every $n\geq i$. Consider the long exact sequence of cohomology,
	\begin{displaymath}
        \mathcal{H}^{n-1}(E_j)\to \mathcal{H}^{n}(E_{j-1})\to \mathcal{H}^{n}(E_{j-1}[n_j])\to \mathcal{H}^{n}(E_j).
    \end{displaymath}
	By localizing at $\mathfrak{p}$, we obtain isomorphisms 
    \begin{displaymath}
        (\mathcal{H}^n(E_{j-1}))_\mathfrak{p}\cong (\mathcal{H}^n(E_{j-1} [n_j] ))_\mathfrak{p}\cong (\mathcal{H}^{n+n_j}(E_{j-1}))_\mathfrak{p}
    \end{displaymath}
    for every $n\geq
	i$. Since $E_{j-1}$ is a bounded complex, we have $\mathcal{H}^{n+kn_i}(E_{j-1})=0$ for
	$k\gg0$. Hence, it follows that
	$(\mathcal{H}^n(E_{j-1}))_\mathfrak{p}\cong(\mathcal{H}^{n+kn_i}(E_{j-1}))_\mathfrak{p}=0$ for every $n\geq
	i$. Consequently, we see that $\mathfrak{p}\notin\phi_{E_{j-1}}(i)$, which completes the proof.
\end{proof}

\begin{lemma}
    \label{lem:tproxy-Krecoll}
    Let $X$ be a Noetherian scheme. An object $P \in D^b_{\operatorname{coh}}(X)$ is $t$-$\otimes$-proxy small if, and only if, the following condition holds:
    \begin{enumerate}
        \item[($\dagger$)] $Q_\lambda(P) \in \overline{\langle Q_\rho(P) \rangle}^{(-\infty,0]}_\odot$.
    \end{enumerate}
\end{lemma}

\begin{proof} 
    By \Cref{prop:aisles_are_compactly_generated}, we know that $\overline{\langle P \rangle}^{(-\infty,0]}_\otimes$ is a compactly generated aisle in $D_{\operatorname{qc}}(X)$. Define $\mathcal{S} \colonequals \overline{\langle P \rangle}^{(-\infty,0]}_\otimes \cap \operatorname{Perf}(X)$. We need to show that $\mathcal{S} \subseteq \langle P \rangle^{[-\infty,0]}_{\otimes}$. Let $\overline{\langle Q_\rho(P) \rangle}^{(-\infty,0]}_\odot$ be the $\odot$-aisle in $K(\operatorname{Inj}(X))$ generated by $Q_\rho(P)$. Fix $S \in \mathcal{S}$. Then, as $Q_\lambda$ is a left adjoint, $S \in \overline{\langle P \rangle}^{(-\infty,0]}_\otimes$ implies that $Q_\lambda(S) \in \overline{\langle Q_\lambda(P) \rangle}^{(-\infty,0]}_\odot$. By the assumption $(\dagger)$, it follows that $Q_\lambda(S) \in \overline{\langle Q_\rho(P) \rangle}^{(-\infty,0]}_\odot$. However, $Q_\lambda(S) \cong Q_\rho(S)\in K(\operatorname{Inj}(X))^c$, which ensures that $Q_\lambda(S) \in \langle Q_\rho(P) \rangle^{[-\infty,0]}_{\odot}$. Applying the localization functor $Q$, we have that $S \in \langle P \rangle^{[-\infty,0]}_{\otimes}$.

    For the converse implication, assume that $P$ is $t$-$\otimes$-proxy small. Then there is $S \in \operatorname{Perf}(X)$ such that $S \in \langle P \rangle^{(-\infty,0)}_\otimes$ and $P \in \overline{\langle S \rangle}^{(-\infty,0)}_\otimes$. Applying $Q_\lambda$ to the latter, we obtain $Q_\lambda(P) \in \overline{\langle Q_\lambda(S) \rangle}^{(-\infty,0)}_\odot = \overline{\langle Q_\rho(S) \rangle}^{(-\infty,0)}_\odot$. Applying $Q_\rho$ to the former we have $Q_\rho(S) \in \langle Q_\rho(P) \rangle^{(-\infty,0)}_\otimes$, which implies $\overline{\langle Q_\rho(S) \rangle}^{(-\infty,0)}_\odot \subseteq \overline{\langle Q_\rho(P) \rangle}^{(-\infty,0]}_\odot$.
\end{proof}

\begin{proof}
    [Proof of \Cref{prop:LCI_characterizations}]
    By \Cref{prop:LCI_iff_proxy_small}, we know that $\eqref{prop:LCI_characterizations1} \iff \eqref{prop:LCI_characterizations3}$. Moreover, it is straightforward to check that $\eqref{prop:LCI_characterizations2} \iff \eqref{prop:LCI_characterizations1}$. So, we only need to show \eqref{prop:LCI_characterizations1} implies \eqref{prop:LCI_characterizations2}. Let $P \in D^b_{\operatorname{coh}}(X)$. By \Cref{prop:LCI_iff_proxy_small}, we know that $P$ is $\otimes$-proxy small. As before, let $\mathcal{T}^{\leq 0}$ be the $\odot$-aisle in $K(\operatorname{Inj}(X))$ generated by $Q_\rho(P)$. In view of \Cref{lem:tproxy-Krecoll}, we need to check that $Q_\lambda(P) \in \mathcal{T}^{\leq 0}$. In view of \Cref{lem:stalk-t-exact}, it suffices to check $(Q_\lambda(P))_x \in \mathcal{T}^{\leq 0}$. But it is not hard to observe that $(Q_\lambda(P))_x \cong Q_\lambda(P_x)$, where $Q_\lambda$ comes from the Krause's recollement for the local ring $\mathcal{O}_{X,x}$. Since $P_x$ is $t$-proxy small over $\mathcal{O}_{X,x}$ by \Cref{prop:local-CI-tproxy}, we have $Q_\lambda(P_x) \in \mathcal{T}^{\leq 0}$ as desired.
\end{proof}

\section{Classification}
\label{sec:classification}

This section proves our main classification results. To start, we need a lemma.

\begin{lemma}
    \label{lem:local-CI-preaisle-thick}
    Let $X$ be a Noetherian scheme which is locally a complete intersection. Then $A[-1] \in \langle A \rangle^{[-\infty,0]}_{\star}$ for any $A \in D_{\operatorname{sg}}(X)$. 
\end{lemma}

\begin{proof}
    By \Cref{lem:stalk-t-exact}, it suffices to prove the statement for the case of $X = \operatorname{Spec}(R)$ where $R$ is a local complete intersection. Indeed, let $\mathcal{T}^{\leq 0}$ be the $\star$-aisle in $S_{\operatorname{qc}}(X)$ (compactly) generated by $A$. Then we have $A[-1] \in \langle A \rangle^{[-\infty,0]}_{\star}$ if, and only if, $A[-1] \in \mathcal{T}^{\leq 0}$. However, the latter condition is equivalent to $A_x[-1] \in \mathcal{T}^{\leq 0}$ for any $x \in X$. Yet, this last condition is implied by $A_x[-1] \in \langle A_x \rangle^{[-\infty,0]}_{\star}$.
    
    Denote by $\pi\colon D^b_{\operatorname{coh}}(R) \to D_{\operatorname{sg}}(R)$ the natural functor. Choose $E \in D^b_{\operatorname{coh}}(R)$ such that $\pi(E)\cong A$. Consider the rotations of the distinguished triangles appearing in the proof of \Cref{prop:local-CI-tproxy},
    \begin{displaymath}
        E_{i}[-1] \to E_{i-1} \to E_{i-1} [n_i] \to E_{i},
    \end{displaymath}
    where $n_i>0$, $E_0 = E$, and $E_t\in \operatorname{Perf}(R)$ for some $t \geq 0$. Applying $\pi$, we obtain distinguished triangles in $D_{\operatorname{sg}}(R)$, 
    \begin{displaymath}
        A_{i}[-1] \to A_{i-1} \to A_{i-1} [n_i] \to A_{i}.
    \end{displaymath}
    We claim that $A_i[-1] \in \langle A_i \rangle^{[-\infty,0]}$. This can be shown by backwards induction on $i=t,t-1,\ldots,0$. The case $i=t$ is clear as $E_t$ is compact, i.e.\ $A_t\cong 0$ in $D_{\operatorname{sg}}(R)$. For the induction step, the distinguished triangle above gives us $A_{i-1}[-1] \in \langle A_{i-1},A_i[-2] \rangle^{[-\infty,0]}$. By induction, we have $A_i[-2] \in \langle A_{i} \rangle^{[-\infty,0]} \subseteq \langle A_{i-1} \rangle^{[-\infty,0]}$, so that $A_{i-1}[-1] \in \langle A_{i-1} \rangle^{[-\infty,0]}$.
\end{proof}

\begin{reminder}
    Let $X$ be a Noetherian scheme which is locally a complete intersection. In view of \Cref{prop:LCI_characterizations} and \Cref{lem:local-CI-preaisle-thick}, \Cref{thm:injective_mapping_for_t_proxy_tensor_aisles_scheme} provides the following injective map:
    \begin{displaymath}
            \begin{aligned}
                \Theta \colon & \left\{ \otimes\textrm{-suspended subcategories }\mathcal{S}\subseteq D^b_{\operatorname{coh}}(X) \right\}
                \\&\to \left\{ \star\textrm{-thick subcategories of } D_{\operatorname{sg}}(X) \right\} \times \left\{ \textrm{Thomason filtrations on } X \right\}.
            \end{aligned}
    \end{displaymath}
    Following \cite{Stevenson:2014b}, we define the \textbf{support} of an object $A \in D_{\operatorname{sg}}(X)$ by 
    \begin{displaymath}
        \operatorname{supp}_{\operatorname{sg}}(A) \colonequals \{x \in X \mid A_x \text{ is not perfect}\} \subseteq \operatorname{sing}(X).
    \end{displaymath}
    If $\mathcal{A}\subseteq D_{\operatorname{sg}}(X)$, set $\operatorname{supp}_{\operatorname{sg}}(A) = \bigcup_{A \in \mathcal{A}}\operatorname{supp}_{\operatorname{sg}}(A)$. Given a specialization closed subset $W$ of $\operatorname{sing}(X)$, we have the $\star$-thick subcategory of $D_{\operatorname{sg}}(X)$,
    \begin{displaymath}
        \operatorname{supp}^{-1}_{\operatorname{sg}}(W) = \{A \in D_{\operatorname{sg}}(X) \mid \operatorname{supp}_{\operatorname{sg}}(A) \subseteq W\}.
    \end{displaymath}
\end{reminder}

\begin{lemma}
    \label{lem:singular-lift-support_hyp}
    Let $X$ be a separated Gorenstein Noetherian scheme. Let $A \in D_{\operatorname{sg}}(X)$ and set $W = \operatorname{supp}_{\operatorname{sg}}(A)$. Then $W$ is closed in $X$ and there is $M \in D^b_{\operatorname{coh}}(X)$ representing $A$ such that $\operatorname{supp}(M) \subseteq W$.
\end{lemma}

\begin{proof}
    That $W$ is closed is proved in \cite[Proposition 7.5]{Stevenson:2014}. The second claim follows from \cite[Proposition 6.9]{Krause:2005} (see also the discussion following it) but we spell out some details. Let $U\colonequals X \setminus W$. The open immersion $U \to X$ induces a localizing sequence 
    \begin{displaymath}
        S_{W,\operatorname{qc}}(X) \to S_{\operatorname{qc}}(X) \to S_{\operatorname{qc}}(U)
    \end{displaymath}
    where $S_{W,\operatorname{qc}}(X)$ consists of acyclic complexes $B$ such that $\operatorname{supp}_{i \in \mathbb{Z}}B^i \subseteq W$. Since $A_x = 0$ in $S_{\operatorname{qc}}(X)$ for any $x \in U$, $A$ is sent to the zero object in $S_{\operatorname{qc}}(U)$, and thus belongs to $S_{W,\operatorname{qc}}(X)$. By the discussion following \cite[Proposition 6.9]{Krause:2005}, $S_{W,\operatorname{qc}}(X)$ is in the image under $I_\lambda$ of the subcategory $K_{W,\operatorname{qc}}(\operatorname{Inj}(X))$ of $K(\operatorname{Inj}(X))$ defined similarly. Then $A$ can be lifted along $I_\lambda$ to an injective resolution of an object of $D^b_{\operatorname{coh}}(X)$ which is supported inside $W$.
\end{proof}

\begin{theorem}
    \label{thm:image_assignment_tstr_lci_general}
    Let $X$ be a Noetherian scheme which is locally a complete intersection. Then the assignment $\Theta$ induces a bijection between:
    \begin{enumerate}
        \item $\otimes$-suspended subcategories of $D_{\operatorname{coh}}^b(X)$, and
        \item pairs $(\mathcal{A},\phi)$ of a $\star$-thick subcategory $\mathcal{A}$ of $D_{\operatorname{sg}}(X)$ and a Thomason filtration $\phi$ on $X$ such that $\operatorname{supp}_{\operatorname{sg}}(\mathcal{A}) \subseteq \bigcup_{n \in \mathbb{Z}}\phi(n)$.
    \end{enumerate}
\end{theorem}

\begin{proof}
    By \Cref{prop:image_assignment_tstr_general} and the discussion above, the image of $\Theta$ consists precisely of those pairs $\mathcal{A}$ and $\phi$ such that for each $A \in \mathcal{A}$ there is $M \in D_{\operatorname{coh}}^b(X)$ with $\operatorname{supp}(\mathcal{H}^i(M)) \subseteq \phi(i)$ for all $i \in \mathbb{Z}$ such that $I_\lambda Q_\rho(M) \in \mathcal{A}$ and $A \in \langle I_\lambda Q_\rho(M) \rangle^{[-\infty,0]}_\star$. Therefore, our goal is to show that $\operatorname{supp}_{\operatorname{sg}}(\mathcal{A}) \subseteq \bigcup_{n \in \mathbb{Z}}\phi(n)$ is both a necessary and a sufficient condition for this to hold.

    The condition is clearly necessary, indeed, let $A \in \mathcal{A}$ be such that $\operatorname{supp}_{\operatorname{sg}}(A) \not\subseteq \bigcup_{n \in \mathbb{Z}}\phi(n)$. Then for any $M \in D^b_{\operatorname{coh}}(X)$ such that $\operatorname{supp}(M) \subseteq \bigcup_{n \in \mathbb{Z}}\phi(n)$ the object $A$ cannot belong to $\langle I_\lambda Q_\rho(M) \rangle^{[-\infty,0]}_\star \subseteq \operatorname{supp}^{-1}_{\operatorname{sg}}(\operatorname{supp}_{\operatorname{sg}}(I_\lambda Q_\rho(M))$.
    
    Now assume that $\operatorname{supp}_{\operatorname{sg}}(\mathcal{A}) \subseteq \bigcup_{n \in \mathbb{Z}}\phi(n)$ holds, and let $A \in \mathcal{A}$. There is $M \in D^b_{\operatorname{coh}}$ such that $M$ lifts $A$ along $D^b_{\operatorname{coh}}(X) \to D_{\operatorname{sg}}(X)$ and $\operatorname{supp}(M) \subseteq \operatorname{supp}_{\operatorname{sg}}(A)$, this uses \Cref{lem:singular-lift-support_hyp}. Then $\operatorname{supp}(M) \subseteq \bigcup_{n \in \mathbb{Z}}\phi(n)$, and since $\operatorname{supp}(M)$ is closed, there is $n \in \mathbb{Z}$ such that $\operatorname{supp}(M) \subseteq \phi(n)$. Therefore, there must be an even integer $m \in \mathbb{Z}$ such that the condition on cohomology is satisfied for $M[m]$. Clearly, $I_\lambda Q_\rho M[m] \in \mathcal{A}$. On the other hand, \Cref{lem:local-CI-preaisle-thick} yields $A = I_\lambda Q_\rho M \in\langle M[m] \rangle^{[-\infty,0]}_\star$, so the condition is satisfied for $M[m]$.
\end{proof}

\begin{remark}
    \label{rmk:special-cases-classification}
    \hfill
    \begin{enumerate}
        \item \label{rmk:special-cases-classification1} The $\otimes$-thick subcategories of $D_{\operatorname{coh}}^b(X)$ correspond under \Cref{thm:image_assignment_tstr_lci_general} precisely to pairs $(\mathcal{A},\phi)$ such that $\phi$ is constant, that is, there is a specialization closed subset $V$ of $X$ such that $\phi(n) = V$ for all $n \in \mathbb{Z}$. The last condition then simplifies to $\operatorname{supp}_{\operatorname{sg}}(\mathcal{A}) \subseteq V$.
        \item \label{rmk:special-cases-classification2} The $\otimes$-suspended subcategories of $D_{\operatorname{coh}}^b(X)$ which contain $\mathcal{O}_X$ in degree 0 correspond under \Cref{thm:image_assignment_tstr_lci_general} precisely to pairs $(\mathcal{A},\phi)$ such that $\phi(0) = X$. The last condition then becomes vacuous and thus all such pairs are in the image of $\Theta$.
    \end{enumerate}
\end{remark}

\begin{proof}
    [Proof of \Cref{thm:image_assignment_tstr_lci_hyp}]
    This follows directly by combining \Cref{thm:image_assignment_tstr_lci_general} with \cite[Corollary 7.9]{Stevenson:2014b}.
\end{proof}

\subsection{Complete intersection schemes}
\label{sec:classification_lci}

Apart from schemes with hypersurface singularities, \cite{Stevenson:2014,Stevenson:2014b} provided an explicit classification of $\star$-thick subcategories of $D_{\operatorname{sg}}(X)$ in the more general case of complete intersection schemes. We recall the precise setup in which this was done and show that \Cref{thm:image_assignment_tstr_lci_general} can be made into an explicit classification in this case as well.

\begin{setup}
    \label{setup:orlov-stevenson}
    Let $X$ satisfy the following conditions.
    \begin{enumerate}        
        \item \label{setup:orlov-stevenson1} $X$ is a complete intersection, that is, there is a separated regular Noetherian scheme $T$ of finite Krull dimension, a vector bundle $\mathcal{E}$ on $T$ of finite rank, and a global section $t \in H^0(T,\mathcal{E})$ such that $X$ is the zero subscheme of $t$ in $T$.
        \item \label{setup:orlov-stevenson2} $X$ has the resolution property.
        \item \label{setup:orlov-stevenson3} The line bundle $\mathcal{O}_\mathcal{E}(1)$ is ample.
    \end{enumerate}
    Following \cite{Orlov:2006}, see also the generalization of \cite[Appendix A]{Burke/Walker:2015} and discussion in \cite[\S 2]{Stevenson:2014}, $X$ satisfying the conditions \eqref{setup:orlov-stevenson1} and \eqref{setup:orlov-stevenson2} above admits scheme morphisms $X \xleftarrow{p} Z \xrightarrow{i} Y$ such that:
    \begin{enumerate}
        \item $Z$ is the projective bundle of the normal bundle ${\mathcal{N}}_{X/T}$ in $T$ and $p$ is the canonical projection
        \item $Y$ is a hypersurface scheme and $i$ is a closed immersion
        \item the induced functor $\mathbf{R} i_\ast \mathbf{L} p^\ast \colon D_{\operatorname{coh}}^b(X) \to D_{\operatorname{coh}}^b(Y)$ is fully faithful and further induces an equivalence $\mathbf{R}i_\ast \mathbf{L} p^\ast\colon D_{\operatorname{sg}}(X) \to D_{\operatorname{sg}}(Y)$.
    \end{enumerate}
\end{setup}

\begin{theorem}
    \label{thm:image_assignment_tstr_ci}
    Let $X$ be as in \Cref{setup:orlov-stevenson}. Then $\Theta$ induces a bijection between
    \begin{enumerate}
        \item $\otimes$-suspended subcategories of $D_{\operatorname{coh}}^b(X)$, and
        \item pairs $(W,\phi)$ of a specialization closed subset $W$ of $\operatorname{sing}(Y)$ and a Thomason filtration $\phi$ on $X$ such that $p i^{-1}(W) \subseteq \bigcup_{n \in \mathbb{Z}}\phi(n)$.
    \end{enumerate}
\end{theorem}

\begin{proof}
    This follows by combining \Cref{thm:image_assignment_tstr_lci_general} with \cite[Theorem 4.9]{Stevenson:2014}. Indeed, the specialization closed subsets $W$ of $\operatorname{sing}(Y)$ correspond to $\star$-thick subcategories of $D_{\operatorname{sg}}(X)$ via the assignment $W \mapsto \{A \in D_{\operatorname{sg}}(X) \mid \operatorname{supp}_{\operatorname{sg,Y}}(\mathbf{R}i_\ast \mathbf{L} p^\ast A) \subseteq W\}$. It remains to check that
    \begin{displaymath}
        p i^{-1}(\operatorname{supp}_{\operatorname{sg,Y}}(\mathbf{R}i_\ast \mathbf{L} p^\ast A)) = \operatorname{supp}_{\operatorname{sg,X}}(A).
    \end{displaymath} 
    In another words, we need to check for any $x \in X$ and $y \in Y$ such that $x = pi^{-1}(y)$ that $A_x$ is perfect over $\mathcal{O}_{X,x}$ if, and only if, $(\mathbf{R}i_\ast \mathbf{L} p^\ast A)_y$ is perfect over $\mathcal{O}_{Y,y}$. Since $p$ is surjective, it is easy to see that $A_x \in \operatorname{Perf}(\mathcal{O}_{X,x})$ if, and only if, $( \mathbf{L} p^\ast A)_z \cong A_x \otimes_{\mathcal{O}_{X,x}} {O}_{Z,z} \in \operatorname{Perf}(\mathcal{O}_{Z,z})$ for all $z \in p^{-1}(x)$. Next, by the construction of \cite{Orlov:2006} the morphism $i\colon Z \to Y$ is a closed immersion of finite Tor dimension, so $\mathbf{L} p^\ast A)_z \in \operatorname{Perf}(\mathcal{O}_{Z,z})$ if, and only if, $(\mathbf{R}i_\ast \mathbf{L} p^\ast A)_z \in \operatorname{Perf}(\mathcal{O}_{Y,z})$.
\end{proof}

\begin{setup}
    \label{setup:takahashi}
    Let $(R,V)$ be one of the following pairs of a Noetherian ring $R$ and a topological space $V$:
    \begin{enumerate}
        \item \label{setup:takahashi1} $R$ is locally a hypersurface and $V = \operatorname{Sing}(R)$.
        \item \label{setup:takahashi2} $R = S/(\mathbf{a})$ for a regular sequence $\mathbf{a} = (a_1,\ldots,a_n)$ in a regular ring $S$ of finite Krull dimension and $V$ is the singular locus of the zero subscheme $Y$ of the section $a_1 x_1 + \cdots + a_n x_n \in \Gamma(\mathbb{P}^{n-1}_S,\mathcal{O}_{\mathbb{P}^{n-1}_S}(1))$.
    \end{enumerate}
    In the situation (2), set $Z = \mathbb{P}^{n-1}_R$ and denote by $p\colon Z \to \operatorname{Spec}(R)$ the canonical projection. Following Stevenson \cite[\S 10]{Stevenson:2014b}, there is an induced closed immersion $i\colon Z \to Y$ and the morphisms $X \xleftarrow{p} Z \xrightarrow{i} Y$ coincide with the ones of \Cref{setup:orlov-stevenson} for $X = \operatorname{Spec}(R)$. In situation \eqref{setup:takahashi1}, we set both $p$ and $i$ to be the identity morphisms on $\operatorname{Spec}(R)$.
\end{setup}

\begin{corollary}\label{cor:affine_classification}
    Let $R$ be as in \Cref{setup:takahashi}. Then $\Theta$ induces a bijection between
    \begin{enumerate}
        \item suspended subcategories of $D_{\operatorname{coh}}^b(R)$, and
        \item pairs $(W,\phi)$ of a specialization closed subset $W$ of $V$ and a Thomason filtration $\phi$ on $\operatorname{Spec}(R)$ such that $p i^{-1}(W) \subseteq \cup_{n \in \mathbb{Z}}\phi(n)$.
    \end{enumerate}
\end{corollary}

\begin{proof}
    This is a special case of \Cref{thm:image_assignment_tstr_lci_hyp} and \Cref{thm:image_assignment_tstr_ci} applied to $X = \operatorname{Spec}(R)$. Indeed, if $R$ is locally a hypersurface then the claim is just \Cref{thm:image_assignment_tstr_lci_hyp} restricted to affine schemes. In the situation \eqref{setup:takahashi2} of \Cref{setup:takahashi}, $X$ satisfies \Cref{setup:orlov-stevenson}, see \cite[\S 2]{Stevenson:2014}. Furthermore, as $X$ is affine, any suspended subcategory of $D_{\operatorname{coh}}^b(R)$ is automatically $\odot$-suspended, and so the claim is a special case of \Cref{thm:image_assignment_tstr_ci}.
\end{proof}

We remark that \Cref{cor:affine_classification} specializes to the following recent result of Takahashi. 

\begin{corollary}
    [cf.\ {\cite[Theorem 1.4]{Takahashi:2025}}]
    \label{rmk:generalize_takahashi}
     Let $R$ be as in \Cref{setup:takahashi}. Then $\Theta$ induces a bijection between
    \begin{enumerate}
        \item suspended subcategories of $D_{\operatorname{coh}}^b(R)$ containing $R$, and
        \item pairs $(W,\phi)$ of a specialization closed subset $W$ of $V$ and a Thomason filtration $\phi$ on $\operatorname{Spec}(R)$ such that $\phi(0) = \operatorname{Spec}(R)$.
    \end{enumerate}
\end{corollary}

\begin{proof}
    In view of \eqref{rmk:special-cases-classification2} in \Cref{rmk:special-cases-classification}, suspended subcategories of $D_{\operatorname{coh}}^b(R)$ containing $R$ correspond via \Cref{cor:affine_classification} to pairs $(W,\phi)$ as in \eqref{rmk:special-cases-classification2} of \Cref{rmk:special-cases-classification} such that $\phi(0)= \operatorname{Spec}(R)$. Note that in this situation, the extra condition `$p i^{-1}(W) \subseteq \cup_{n \in \mathbb{Z}}\phi(n)$' of \eqref{rmk:special-cases-classification2} in \Cref{rmk:special-cases-classification} becomes vacuous.
\end{proof}

\begin{remark}
    In \cite[Theorem 1.4]{Takahashi:2025}, the classification is formulated slightly differently, using order preserving maps from $\operatorname{Spec}(R)$ to $\mathbb{N} \cup \{\infty\}$. It is not difficult to match this with our formulation using Thomason filtrations, see e.g. the discussion \cite[\S 2.4]{Hrbek/Nakamura/Stovicek:2024}.
\end{remark}

\bibliographystyle{alpha}
\bibliography{mainbib}

\end{document}